\documentclass[a4paper]{article}

\usepackage[utf8]{inputenc}
\usepackage{amsmath}
\usepackage{amsfonts}
\usepackage{amsthm}
\usepackage{array}
\usepackage[hidelinks]{hyperref}
\usepackage[english]{babel}
\usepackage{float}

\newtheorem{Lem}{Lemma}
\newtheorem{Def}{Definition}
\newcommand{\enstq}[2]{\left\{#1\mathrel{}\middle|\mathrel{}#2\right\}}
\newcommand{\norm}[1]{\left\|#1\right\|}

\newtheorem{proposition}{Proposition}

\newtheorem{theorem}{Theorem}

\newtheorem{corollary}{Corollary}

\newcommand{\N}{\mathbb{N}}

\newcommand{\R}{\mathbb{R}}
\newcommand{\duality}[2]{\left\langle #1,#2\right\rangle}
\newcommand{\inner}[2]{\left( #1,#2\right)}
\newcommand{\abs}[1]{\left\lvert #1 \right\rvert}

\newcommand{\Cinf}{C^{\infty}}
\newcommand{\isdef}{\mathrel{\mathop:}=}

\usepackage{todonotes}
\usepackage{subcaption}
\title{New preconditioners for the Laplace and Helmholtz integral equations on open curves \\
	\vspace{0.2cm}
	{\Large Analytical framework and numerical results} 
}
\date{}
\author{François Alouges   \thanks{Centre de Mathématiques Appliquées (UMR 7641), Ecole Polytechnique, Route de Saclay, 91128  PALAISEAU Cedex.}     \and
	Martin Averseng $^*$}
\begin{document}


\maketitle

\vspace{0.5cm}

\begin{abstract}
	The Helmholtz wave scattering problem by screens in 2D can be recast into first-kind integral equations which lead to ill-conditioned linear systems after discretization. We introduce two new preconditioners, in the form of square-roots of local operators respectively for the corresponding problems with Dirichlet and Neumann conditions on the arc. They generalize the so-called ``analytical" preconditioners available for Lipschitz scatterers. We introduce a functional setting adapted to the singularity of the problem and enabling the analysis of those preconditioners. The efficiency of the method is demonstrated on several numerical examples.
\end{abstract}

\section*{Introduction}

For the numerical resolution of wave scattering problems, a by-now well-established approach is the boundary elements method, which involves the discretization of integral equations leading to linear systems of smaller size when compared to finite element methods, but the corresponding matrices are fully populated. They are solved efficiently by combining iterative methods such as GMRES \cite{saad1986gmres}, with acceleration methods (e.g. the fast multipole method, see \cite{greengard1987fast} and references therein).  The number of iterations in GMRES depends on the condition number of the linear system. The question of preconditioning, by which we mean minimizing this number of iterations, either by formulating well-conditioned integral equations or by finding preconditioners for the linear systems, has attracted a lot of attention since more than two decades. 

In the case of an obstacle with a $C^\infty$ smooth boundary $\Gamma$, general methods to build efficient preconditioners are known \cite{alouges2007stable,alouges2005new,antoine2007generalized,christiansen2002preconditioner,hiptmair2014mesh,steinbach1998construction}, which in most cases can also be applied in theory and/or in practice if the scatterer is only assumed to have Lipschitz regularity. Among them, one efficient approach is the Generalized Combined Source Integral Equation (GCSIE) method \cite{alouges2005new} initiated by Levadoux in \cite{levadoux2001etude}. It involves the inversion of the discretized GCSIE operator 
\[G_k := S_k \tilde{\Lambda}_k + \frac{I_d}{2} + D_k\]
where $S_k$ and $D_k$ stand for the classical single- and double-layer potentials on $\Gamma$, $I_d$ is the identity operator and $\tilde{\Lambda}_k$ is an approximation of the exterior Dirichlet-to-Neumann (DtN) map $\Lambda_k$ for the Helmholtz equation. In the ideal case ${\tilde{\Lambda}_k = \Lambda_k}$, $G_k$ becomes the identity and, consequently, when $\tilde{\Lambda}_k$ is a compact perturbation of $\Lambda_k$, $G_k$ is a compact perturbation of the identity. It is well-known that the Galerkin discretization of such operators leads to well-conditioned linear systems for which the GMRES method converges supra-linearly \cite{axelsson2009equivalent,hiptmair2006operator,steinbach1998construction}. To build suitable candidates for $\tilde{\Lambda}_k$, a generic tool is the theory of pseudo-differential operators (see e.g. \cite{hormander2007analysis}). This has been succesfully applied by Antoine and Darbas in \cite{antoine2007generalized} who proposed the choice
\begin{equation}
\label{defAntoineDarbas}
\tilde{\Lambda}_k := \sqrt{-\Delta_\Gamma - k^2 I_d}\,,
\end{equation}
based on low-order expansions of the symbol of $\Lambda_k$. The authors also introduce an efficient scheme to approximate the square root operator in \eqref{defAntoineDarbas} relying on Pad\'{e} approximants. Their numerical results clearly demonstrate that the preconditioning performances are independent of the discretization parameters and robust in $k$. 

When the scattering object is a ``screen", that is a curve in 2D and a surface in $3D$ \textbf{with edges} (therefore not a Lipschitz domain), the robustness of the precondtioners in the discretization parameters and in $k$ is lost in practice for all of the aforementioned methods. In this work, we propose a generalized version of the square-root preconditioners of Antoine and Darbas on screens in 2D that overcomes this limitation. To fix the notation, let us write the classical first-kind integral equations corresponding to the Dirichlet and Neumann problems as
\[S_k \lambda = u_D\, \quad N_k \mu = u_N\,,\]
where $S_k$ and $N_k$ are the classical single- and hypersingular layer potentials on the screen $\Gamma$ and $u_D$ and $u_N$ are smooth right-hand sides. We define the operators 
\[P_k = \frac{1}{\omega}\sqrt{-(\omega \partial_\tau)^2 - k^2 \omega^2}\,, \quad Q_k = \omega\left[-(\partial_\tau \omega)^2 - k^2 \omega^2\right]^{-\frac{1}{2}}\,,\]
where $\partial_\tau$ is the tangential derivative on $\Gamma$ and $\omega(x)$ is defined for $x \in \Gamma$ as the square root of the distance from $x$ to the edges of $\Gamma$. We present some numerical evidence showing that the preconditioning performances of $P_k$ (resp. $Q_k$) for $S_k$ (resp. $N_k$)  are independent of the discretization parameters, and only very lightly sensitive to the wavenumber $k$. The careful analysis of $P_k$ and $Q_k$, and the proof of the fact that they provide parametrices of $S_k$ and $N_k$ respectively relies on two new classes of pseudo-differential operators on open curves, and is treated in full details in \cite{averseng}. 

To our knowledge, the recent literature mainly contains two other approaches for the problem of preconditioning for Helmholtz screen scattering problems, and the present work is connected to both of them. The first one, see e.g. \cite{hiptmair2014mesh,hiptmair2017closed,jerez2012explicit,ramaciotti2017some,urzua2014optimal}, builds on recently found closed-form expressions for the inverses of the Laplace layer potentials ($k=0$) on flat screens in 2D and 3D (recovered by a unified approach in \cite{gimperlein2019optimal}). Using compact perturbations arguments, those operators are expected to provide efficient preconditioners for the Helmholtz ($k > 0$) layer potentials on arbitrary smooth screens when the wavenumber $k$ is reasonably small. Here, we show that when $\Gamma$ is a flat segment, $P_0$ and $Q_0$ coincide with those exact inverses (up to minor modifications). Our numerical results show that using $P_k$ and $Q_k$ instead of $P_0$ and $Q_0$ for the Helmholtz problem improves greatly the preconditioning performances. 

The second one, developed by Bruno and Lintner \cite{bruno2012second} is a generalization of the Calder\'{o}n preconditioners available for Lipschitz scatterers. In their work, some weighted versions of the Helmholtz layer potentials, $S_{k,\omega}$ and $N_{k,\omega}$, are considered and it is shown that $S_{k,\omega} N_{k,\omega}$ is a second-kind operator, therefore, $S_{k,\omega}$ gives a good preconditioner for $N_{k,\omega}$ and reciprocally. The weighted layer potentials $S_{k,\omega}$ and $N_{k,\omega}$ also appear in the present study as the fundamental objects on which the arguments are developed. Furthermore, we compare numerically our preconditioning method to that of Bruno and Lintner. In our implementation, we find that their method seems slightly more robust in $k$. However, our preconditioners are significantly faster to evaluate due to their quasi-local form. 

The outline of the paper is as follows. We use the first section to fix notation and recall some important results concerning integral equations on screens. In the second section, we introduce the preconditioners for the Laplace problem on a flat screen. The formulas are generalized to the Helmholtz in the third section. In section 4, we describe the weighted Galerkin setup that we use to discretize the integral equations. In section 5, we describe how we build the preconditioners  operators introduced in sections 2 and 3. Finally, in section 6, we show the performance of our preconditioners in a variety of cases and compare them to other ideas of the literature. 

\section{First kind integral equations}

In the following, we consider a smooth, non-intersecting open arc $\Gamma$, that is, a set of the form
\[\Gamma =  \enstq{\gamma(t)}{t \in [-1,1]}\]
where $\gamma: [-1,1] \to \R^2$ is an injective and $C^\infty$ function. The arc $\Gamma$ models a $1$-dimensional screen in a 2D setting. Let $k \geq 0$ be the wavenumber and $G_k$ the Green kernel defined by 
\[G_k(x) = \begin{cases}
-\frac{1}{2\pi} \ln |x| & \textup{if } k = 0\,,\\
\frac{i}{4}H_0^{(1)}(k |x|) & \textup{if } k>0\,,
\end{cases}\]
where $H_0^{(1)}$ is the Hankel function of first kind and $0$ order. The classical single-layer potential, denoted by $S_k$, is defined for $\phi$ is the space $C^\infty_c(\Gamma)$ of compactly supported functions smooth functions on $\Gamma$ by
\[\forall x \in \Gamma\,, \quad S_k \phi(x) = \int_{\Gamma} G_k(x-y) \phi(y) d\sigma(y)\,\]
where $d\sigma$ is the uniform measure on $\Gamma$. Moreover, the Helmholtz hypersingular layer potential $N_k : C^\infty_c(\Gamma) \to \mathcal{D}'(\Gamma)$ is defined by 
\[\duality{N_k \phi}{\psi} := \iint_{\Gamma \times \Gamma} G_k(x - y)\left(\phi'(x) \psi'(y) - k^2 \phi(x) \psi(y) \right)\,d\sigma(x)d\sigma(y)\,,\]
where $\phi'$ and $\psi'$ denote the arclength derivatives of $\phi$ and $\psi$ and $\mathcal{D}'(\Gamma)$ is the set of distributions on $\Gamma$. 

Let us now recall the definitions of the spaces $H^s(\Gamma)$ and $\tilde{H}^s(\Gamma)$ for $s\in \R$, following \cite[chap. 3]{mclean2000strongly}. For this purpose, we consider a closed curve $\tilde{\Gamma}$ such that $\Gamma \subset \tilde{\Gamma}$. A distribution $u$ on $\Gamma$ is said to be in $H^s(\Gamma)$ if there exists an $U \in H^s(\tilde{\Gamma})$ such that $u = U_{|\Gamma}$. Furthermore, $u \in \tilde{H}^s(\Gamma)$ if $u$ is in ${H}^s(\tilde{\Gamma})$ and $\textup{supp}\,u \subset \Gamma$. 
\begin{proposition}[{see \cite[Thm 1.8]{stephan1984augmented} and \cite[Thm 1.4]{wendland1990hypersingular}}]
	The operator $S_k$ has a continuous extension 
	\[S_k : \tilde{H}^{-\frac{1}{2}}(\Gamma) \to H^{1/2}(\Gamma)\,\]
	which is bicontinuous when $k \neq 0$. When $k = 0$, it is bicontinuous if and only if the logarithmic capacity of $\Gamma$ (see e.g. \cite{djikstra}) is not equal to $1$. Similarly, the operator $N_k$ can be extended continuously as
	\[N_k : \tilde{H}^{\frac{1}{2}}(\Gamma) \to H^{-\frac{1}{2}}(\Gamma)\]
	and this extension is bicontinuous. 
\end{proposition} 
It will be assumed in the following that the logarithmic capacity of $\Gamma$ is not $1$. This can always be satisfied by properly rescaling the problem. 
In this work, we are concerned with the resolution of the integral equations
\begin{equation}
\label{integralEquations}
S_k \lambda = u_D\,, \quad N_k \mu = u_N\,.
\end{equation}
where $u_D \in H^{1/2}(\Gamma)$ and $u_N \in H^{-1/2}(\Gamma)$. They are related to the Helmholtz problem in $\mathbb{R}^2 \setminus \Gamma$ with prescribed Dirichlet boundary data $u_D$ (``sound-soft") or Neumann boundary data $u_N$ (``sound-hard") respectively. 

Due to the singularity of the manifold $\Gamma$, the solutions $\lambda$ and $\mu$ of the previous equations have edge singularities even if $u_D$ and $u_N$ are arbitrarily smooth. More precisely, Costabel et al. have shown
\begin{proposition}[{\cite[Cor. A.5.1]{costabel2003asymptotics}}]
	\label{duduch}
	Assume that $u_D$ and $u_N$ are in $C^\infty(\Gamma)$. Then the solutions $\lambda$ and $\mu$ of eq.~\eqref{integralEquations} can be expressed as 
	\[\lambda = \frac{\alpha}{\omega}, \quad \mu = \omega \beta\]
	where $\alpha$ and $\beta$ are in $C^\infty(\Gamma)$ and where 
	$$\omega(x) = \sqrt{{d}(x, \partial \Gamma)}\,.$$
	
\end{proposition}

In addition to the singular nature of the problem, the Galerkin discretization of first-kind integral equations is known to produce ill-conditioned linear systems, see \cite[Sec. 4.5]{sauter2011boundary}. The aim of this paper is to introduce a formalism that enables to resolve the singularity and provide efficient preconditioners for the linear systems.


\section{Laplace equation on a flat segment}

We start with the particular case where $\Gamma$ is the flat segment 
$\Gamma = [-1,1]\times \{0\}$
and furthermore $k=0$. The associated integral equations,
\begin{equation}
\label{symmIntegralEqs}
S_0\lambda = u_D \quad \text{ and } \quad N_0\mu = u_N\,,
\end{equation}
with logarithmic kernels, are the object of a considerable number of papers in the 1990's, for instance \cite{atkinson1991numerical,estrada1989integral,jiang2004second,monch1996numerical,sloan1992collocation,yan1990cosine}. 

More recently, Jerez-Hanckes and Nédélec \cite{jerez2012explicit} have exhibited exact inverses of those operators through explicit variational forms (Prop 3.1 and 3.3) and identities involving tangential square-roots of differential operators with boundary conditions at the end points (Prop 3.10 and subsequent remark). In Gimperlein et al. \cite{gimperlein2019optimal}, those results are recovered with a different method, and it is shown that $S_0$ and $N_0$ are related fractional powers of the Laplace-Beltrami operator on the line $\enstq{(x,0)}{x \in \R}$, by considering the natural extensions of functions by $0$ outside $\Gamma$. 

Here we give a new expression of the inverses of $S_0$ and $N_0$ (\autoref{theorem1} and \autoref{theorem2}), in the form of square roots of suitably \textbf{weighted} tangential differential operators. Those formulas do not involve boundary conditions at the end points nor extensions by $0$ outside $\Gamma$. The proofs of the identities are very simple, nevertheless, it seems, to the best of our knowledge, that they have remained unnoticed so far. Their generalization to the case $k > 0$ (section 3) is the main object of this work.   

\subsection{Analytical setting}

We introduce the Chebyshev polynomials of first and second kinds \cite{mason2002chebyshev}, respectively given by 
\[\forall x \in [-1,1]\, \quad T_n(x) = \cos(n \arccos(x)), \quad U_n(x) = \dfrac{\sin((n+1) \arccos(x))}{\sqrt{1 - x^2}}\,,\]
and we denote by $\omega$ the operator $\omega:u(x) \mapsto \omega(x)u(x)$ with $\omega(x) = \sqrt{1 - x^2}$. We also denote by  
$\partial_x$ the derivation operator. The Chebyshev polynomials satisfy the ordinary differential equations

$$
(1-x^2)\partial_{xx}T_n -x\partial_x T_n +n^2T_n =0
$$
and
$$
(1-x^2)\partial_{xx}U_n -3x\partial_xU_n +n(n+2)U_n =0
$$
which can be rewritten in divergence form as
\begin{eqnarray}
(\omega\partial_x)^2 T_n &=& -n^2T_n\,, \label{cheb1}\\
(\partial_x\omega)^2 U_n &=& -(n+1)^2U_n\,,\label{cheb2}
\end{eqnarray}
where we emphasize that $\partial_x\omega f$ should be understood as the composition of the operators $\partial_x$ and $\omega$ applied to $f$, that is $\partial_x \omega f(x) = \frac{d}{dx}\left( \omega(x) f(x)\right)$. Both $T_n$ and $U_n$ are polynomials of degree $n$, and 
form complete orthogonal families respectively of the Hilbert spaces 
$$L^2_{\frac{1}{\omega}} \isdef \enstq{u \in L^1_\textup{loc}(-1,1)} {\int_{-1}^{1} \dfrac{f^2(x)}{\sqrt{1 - x^2} }dx< + \infty}$$
and 
$$L^2_{\omega} \isdef \enstq{u \in L^1_\textup{loc}(-1,1)} {\int_{-1}^{1} {f^2(x)}{\sqrt{1 - x^2} }dx< + \infty}\,,$$
see e.g. \cite{mason2002chebyshev} and in particular Thm. 5.2. This fact can also be verified directly by remarking that for $(f,g) \in L^2_\frac{1}{\omega} \times L^2_\omega$, the functions 
\[\mathcal{C}f(\theta) := f(\cos\theta)\,, \quad \mathcal{S}g(\theta) := \sin(\theta) g(\cos(\theta))\]
are respectively even and odd functions in $L^2(-\pi,\pi)$. This can be verified simply using the change of variables $x = \cos \theta$. The completeness of the polynomials $T_n$ and $U_n$ in $L^2_\frac{1}{\omega}$ and $L^2_\omega$ are deduced from the completeness of \[\mathcal{C}T_n(\theta) =  \cos(n\theta) \quad \textup{and} \quad \mathcal{S}U_n(\theta) = \sin((n+1)\theta)\]
in the space of $2\pi$-periodic $L^2$ \textbf{even} or \textbf{odd} functions respectively. 

As a consequence, any
$u\in L^2_{\frac{1}{\omega}}$ can be decomposed through the first kind Chebyshev series 
\begin{equation*}
\label{FCseries}
u(x) = \sum_{n=0}^{+\infty} \hat{u}_n T_n(x)
\end{equation*}
where the Fourier-Chebyshev coefficients $\hat{u}_n$ are given by $\hat{u}_n \isdef \dfrac{\inner{u_n}{T_n}_\frac{1}{\omega}}{\inner{T_n}{T_n}_\frac{1}{\omega}}\,.$
with the inner product
\[\inner{u}{v}_{\frac{1}{\omega}} \isdef \frac{1}{\pi}\int_{-1}^{1} \frac{u(x) \overline{v(x)}}{\omega(x)}\,dx\,.\] 
Similarly, any function $v\in L^2_{\omega}$ can be decomposed along the $(U_n)_{n\geq 0}$ as
\[ v(x) = \sum_{n=0}^{+\infty} \check{v}_n U_n(x)\]
where the coefficients $\check{v}_n$ are given by $\check{v}_n \isdef 
\dfrac{\inner{v}{U_n}_{\omega}}{\inner{U_n}{U_n}_\omega}$ with the inner product
\[\inner{u}{v}_\omega = \int_{-1}^1 u(x) v(x) \omega(x)\,dx\,.\] 
Those properties can be used to define Sobolev-like spaces. 
\begin{Def}
	For all $s \geq 0$, we define 
	\[T^s = \enstq{ u \in L^2_\frac{1}{\omega}}{ \sum_{n=0}^{+\infty}(1+n^2)^s \abs{\hat{u}_n}^2 < + \infty}.\]
	Endowed with the scalar product
	\[\inner{u}{v}_{T^s} = \hat{u}_0 \overline{\hat{v}_0} + \frac{1}{2}\sum_{n=1}^{+\infty}(1+n^2)^s\hat{u}_n \overline{\hat{v}_n},\]
	$T^s$ is a Hilbert space for all $s \geq 0$. 
	Similarly, we set
	\[U^s = \enstq{u \in L^2_\omega}{ \sum_{n=0}^{+\infty} (1 + n^2)^s\abs{\check{u}_n}^2}\]
	which is a Hilbert space for the scalar product
	\[\inner{u}{v}_{U^s} = \frac{1}{2} \sum_{n=0}^{+ \infty} (1 + (n+1)^2)^s\check{u}_n\overline{\check{v}_n}.\]
\end{Def}
One can extend the definitions of $T^s$ ad $U^s$ for $s\in \R$, in which case they form interpolating scales of Hilbert space, see \cite[Section 1]{averseng} for details. Letting $H^s_e \oplus H^s_o$ be the partition of the Sobolev space of $2\pi$-periodic functions $H^s_{\textup{per}}$ into even and odd functions, it is easy to check that the pullbacks
\[\mathcal{C} : T^s \to H^s_e \quad \textup{and} \quad \mathcal{S} : U^s \to H^s_o\] 
are isomorphisms. It follows that the inclusions 
\[T^{s'} \subset T^s, \quad U^{s'} \subset U^s, \quad s' > s\]
are compact. 
Denoting by $T^\infty = \cap_{s \geq 0} T^s$ and similarly for $U^\infty$, it is shown in \cite[Lem. 7]{averseng} that
\begin{Lem}
	\label{LemTinfCinf}
	\[T^{\infty} = U^{\infty} = C^{\infty}([-1,1])\,.\]
\end{Lem}
\noindent For $s = \pm \frac{1}{2}$, those spaces 
have been analyzed (with different notation) e.g. in \cite{jerez2012explicit} and verify
\begin{equation}
\label{lemJerez1}
T^{-1/2} = {\omega} \tilde{H}^{-1/2}(-1,1), \quad T^{1/2} = H^{1/2}(-1,1)\,,
\end{equation}
\begin{equation}
\label{lemJerez2}
U^{-1/2} = H^{-1/2}(-1,1), \quad U^{1/2} = \frac{1}{\omega }\tilde{H}^{1/2}(-1,1)\,.
\end{equation}

\subsection{Laplace single-layer equation}

We start with the single-layer integral equation $S_0\lambda = g$, with $\lambda \in \tilde{H}^{-1/2}(\Gamma)$, that is
\begin{equation}
-\frac{1}{2\pi}\int_{-1}^{1} \log|x-y| \lambda(y) = g(x), \quad \forall x\in (-1,1)\,.\label{Slambda}
\end{equation} 
The following result is the fundamental tool for studying this equation:

\begin{Lem}[See e.g. {\cite[Thm 9.2]{mason2002chebyshev}}]
	For all $n\in \mathbb{N}$, we have
	\[-\frac{1}{2\pi}\int_{-1}^{1} \frac{\ln|x-y|}{\sqrt{1 - y^2}}T_n(y)dy = \sigma_n T_n(x)\]
	where
	\[\sigma_n = \begin{cases}
	\dfrac{\ln(2)}{2} & \text{if } n=0\\
	\\
	\dfrac{1}{2n} & \text{otherwise}.
	\end{cases}\]
	\label{STn}
\end{Lem}

Using the decomposition of $g$ on the basis $(T_n)_n$, we see at once that the solution $\lambda$ to equation \eqref{Slambda} admits the expansion 
\begin{equation}
\lambda(x) = \frac{1}{\sqrt{1-x^2}}\sum_{n=0}^{+ \infty} \frac{\hat{g}_n}{\sigma_n} T_n(x)\,.
\label{expansionLambda}
\end{equation}
As a corollary, we obtain by an alternative proof the result of Costabel et al. cited in Proposition \ref{duduch} in this particular case:
\begin{corollary}
	\label{CorSingularity}
	If the data $g$ is in $C^{\infty}([-1,1])$, the solution $\lambda$ to the equation 
	\[S_0\lambda = g\]
	is of the form 
	\[\lambda(x) = \dfrac{\alpha(x)}{\sqrt{1-x^2}}\]
	with $\alpha \in C^{\infty}([-1,1])$.  
\end{corollary}

\begin{proof}
	Let $\alpha (x)= \sqrt{1 - x^2}\,\lambda(x)$ where $\lambda$ is the solution of $S_0\lambda = g$.  
	By \autoref{LemTinfCinf}, if $g \in C^{\infty}([-1,1])$, then $g \in T^{\infty}$, and by equation 
	\eqref{expansionLambda}, 
	\[ \hat{\alpha}_n = \frac{\hat{g}_n}{\sigma_n}\,,\]
	from which we deduce that $\alpha$ also  belongs to $T^{\infty} = C^{\infty}([-1,1])$. 
\end{proof}

\noindent Following \cite{bruno2012second}, we introduce the weighted single layer operator as the operator 
that appears in \autoref{STn}.
\begin{Def}
	Let $S_{0,\omega}$ be the weighted single layer operator defined by
	\[{S_{0,\omega}} : \quad {\alpha \in \Cinf([-1,1])}\mapsto {-\dfrac{1}{2\pi}}{\int_{-1}^1\dfrac{\ln|x-y|}{\omega(y)} \alpha(y)dy}\,.\]
\end{Def}
\noindent Lemma \ref{STn} can be restated as saying that the Chebyshev polynomials $T_n$ are eigenfunctions of $S_{0,\omega}$. To obtain the solution of the single-layer( integral equation \eqref{Slambda}, we thus solve the \textbf{weighted} single-layer integral equation
\begin{equation}
S_{0,\omega} \alpha = u_D\,,
\label{Somegaalpha}
\end{equation}
and let $\lambda = \frac{\alpha}{\omega}$, which indeed belongs to $\tilde{H}^{-1/2}(\Gamma)$ by eq.~\eqref{lemJerez1}. 

Comparing the eigenvalues of $S_{0,\omega}$ and $-(\omega \partial_x)^2$, we directly obtain the following result:

\begin{theorem} 
	\label{theorem1}
	There holds
	\[S_{0,\omega}^{-1} = 2\sqrt{-(\omega \partial_x)^2} + \frac{2}{\ln(2)}\pi_0\,\]
	or equivalently, 
	\[S_0^{-1} = \frac{2}{\omega} \sqrt{-(\omega \partial_x)^2} + \frac{2}{\ln(2)} \frac{\pi_0}{\omega}\,. \]
	where $\pi_0$ is the $L^2_\frac{1}{\omega}$ orthogonal projection on $T_0$, that is
	\[\quad \pi_0 \phi =\frac{1}{\pi} \int_{-1}^{1} \frac{\phi(y) dy}{\omega(y)}\,.\]
\end{theorem}
The operator $-(\omega \partial_x)^2$ being self-adjoint (with domain $T^1$), its square-root can be defined via the spectral theorem, see for example \cite[Def. 10.5]{hall2013quantum}. In later occurences, the square-root will also be applied to self-adjoint operators whose spectrum may contain negative values. In this case, we take the standard definition of the square root, with $\sqrt{-r} = ir$ for all $r > 0$. 
\begin{proof}
	We simply notice that for all $n \in \N$, $-(\omega \partial_x)^2 T_n = n^2 T_n$, thus
	\[\sqrt{-(\omega \partial_x)^2} T_n = nT_n\,.\] 
	Therefore, 
	\begin{eqnarray*}
		\left(2\sqrt{-(\omega \partial_x)^2} + \frac{2}{\ln(2)} \pi_0\right) T_n &=& \begin{cases}
			2n\,T_n & \textup{ if } n \neq 0 \\
			\frac{2}{\ln(2)}T_0 & \textup{ otherwise }
		\end{cases}\\
		& = &\frac{1}{\sigma_n}T_n\,. \quad \quad \hspace{0.11cm} 
	\end{eqnarray*}
\end{proof}
%

Since the operator $2\sqrt{-(\omega \partial_x)^2} + \frac{2}{\ln(2)}\pi_0$ is the inverse of $S_{0,\omega}$, it can be used as an efficient preconditioner for the weighted integral equation~\eqref{Somegaalpha}. 
\subsection{Laplace hypersingular equation} 

We now turn our attention to the Laplace hypersingular equation 
\begin{equation}
N_0\mu = g \quad \textup{in } H^{1/2}(\Gamma) \,.
\label{Nmu}
\end{equation} 
\noindent Similarly to the previous section and again following \cite{bruno2012second}, we consider the weighted operator $N_{0,\omega} \isdef N_0 \omega$. 
We can get the solution to equation \eqref{Nmu} by solving the \textbf{weighted} hypersingular integral equation
\begin{equation}
N_{0,\omega} \beta = u_N,
\label{Nomegabeta}
\end{equation}
and letting $\mu = \omega \beta$. 
We now show that $N_{0,\omega}$ can be analyzed using this time the spaces $U^s$. 
\begin{Lem}
	\label{lemIPP}
	For any $\beta$, $\beta' \in U^{\infty}$, one has 
	\[\duality{N_{0,\omega} \beta}{ \beta'}_\omega = \duality{S_{0,\omega} \omega \partial_x \omega \beta}{\omega \partial_x \omega \beta'}_\frac{1}{\omega}.\]
\end{Lem}
\begin{proof}
	We use the well-known integration by parts formula
	\[\duality{N_0 u}{v} = \duality{S_0\partial_x u}{\partial_x v},\]
	valid when $u$ and $v$ are regular enough and vanish at the endpoints of the segment. 
	For a smooth $\beta$, we thus have
	\[ \duality{N_0 (\omega \beta)}{ (\omega \beta')} = \duality{S_0 \partial_x(\omega \beta)}{\partial_x (\omega \beta')}\] 
	which implies the claimed identity. 
\end{proof}
\begin{Lem}
	For all $n \in \N$, there holds
	\[N_{0,\omega} U_n = \frac{n+1}{2}U_n.\]	
	\label{NUn}
\end{Lem}
\begin{proof}
	From the identity $\partial_x T_{n+1} = (n+1)U_n$ and eq.~\eqref{cheb1}, we obtain
	\begin{equation*}
	\omega \partial_x \omega U_n = -(n+1) T_{n+1}.
	\end{equation*}
	Therefore, by \autoref{lemIPP}
	\begin{eqnarray*}
		\duality{N_{0,\omega} U_m}{U_n}_\omega & = & (n+1)(m+1)\duality{S_{0,\omega} T_{m+1}}{T_{n+1}}_\frac{1}{\omega}\nonumber\\
		&=& \delta_{m=n} \frac{n+1}{2}.
	\end{eqnarray*} 
	
\end{proof}
The identity $-(\partial_x\omega)^2 U_n = (n+1)^2 U_n$ now leads to the following result:
\begin{theorem} 
	\label{theorem2}
	\[N_{0,\omega}^{-1} = 2\left[{-(\partial_x\omega)^2}\right]^{-\frac{1}{2}} \quad \textup{in } U^{-\infty}\,\]
	or equivalently, 
	\[N_0^{-1} = 2 \omega \left[-(\partial_x \omega)^2\right]^{-1/2}\,.\]
\end{theorem}
Here again the inverse square root is defined by functional calculus, which is possible since $-(\partial_x \omega)^2$ is self-adjoint (with domain $U^1$).

\section{Helmholtz equation}

In this section, we aim at generalizing the preceding analysis to the case of Helmholtz equation ($k > 0$) on $\mathbb{R}^2\setminus \Gamma$ 
with $\Gamma = [-1,1]\times\{0\}$, based on the explicit formulas presented in the previous section. In an analogous fashion as above, let $S_{k,\omega} \isdef S_k \frac{1}{\omega}$ and 
$N_{k,\omega} \isdef N_k \omega$. The following commutation holds:	
\begin{theorem}
	\label{commutations}
	\[S_{k,\omega} \left[-(\omega \partial_x)^2 - k^2\omega^2\right] =  \left[-(\omega \partial_x)^2 - k^2\omega^2\right]S_{k,\omega} \quad \textup{in } T^{-\infty}\]
\end{theorem}
\begin{proof}
	Let us compute 
	\[\left(S_{k,\omega} (\omega_x \partial_x)^2 u\right)(x)= \int_{-1}^{1} G_k(x-y) \frac{(\omega_y \partial_y)^2u(y)  }{\omega(y)},\]
	where we use the notation $\omega_y$ and $\partial_y$ to emphasize the dependence in the variable $y$. 
	Since $(\omega \partial_x)^2$ is symmetric with respect to the bilinear form $\inner{\cdot}{\cdot}_{\frac{1}{\omega}}$, we have
	\[\left(S_{k,\omega} (\omega_x \partial_x)^2 u\right)(x)= \int_{-1}^{1} (\omega_y \partial_y)^2[G_k(x-y)] \frac{u(y)  }{\omega(y)},\]

	Thus, 
	\[\left(\left(S_{k,\omega} (\omega \partial_x)^2 - (\omega \partial_x)^2 S_{k,\omega}\right)u\right)(x) = \int_{-1}^{1} \frac{D_k(x,y)u(y)}{\omega(y)},\]
	where $D_k(x,y) \isdef \left[(\omega_y \partial_y)^2 - (\omega_x \partial_x)^2\right] \left[G_k(x-y)\right]$. 
	A simple computation leads to 
	\[D_k(x,y) = \partial_{xx}G_k(x-y) (\omega^2_y - \omega^2_x) + \partial_x G_k(x-y)(y + x).\]
	Since $G_k$ is a solution of the Helmholtz equation, we have for all $(x \neq y) \in \R$ 
	\[\partial_x G_k(x-y) = (y-x)(\partial_{xx}G_k(x-y) + k^2G(x-y)),\]
	thus
	\[D_k(x,y) = \partial_{xx}G_k(x-y)\left(\omega^2_y - \omega_x^2 + y^2 - x^2\right) + k^2(y^2 - x^2)G_k(x-y) . \]
	A careful analysis shows that no Dirac mass appears in the previous formula.
	Note that $y^2 - x^2 = \omega_x^2 - \omega_y^2$ so the first term vanishes and we find
	\[S_{k,\omega} (\omega \partial_x)^2 - (\omega \partial_x)^2 S_{k,\omega} =  k^2\left(\omega^2 S_{k,\omega} -S_{k,\omega} \omega^2 \right)\]
	as claimed. 
\end{proof}
\noindent There also holds the following identity:
\[N_{k,\omega} \left[-(\partial_x \omega)^2 - k^2\omega^2\right] =  \left[-(\partial_x \omega)^2 - k^2\omega^2\right]N_{k,\omega} \quad \textup{in } U^{-\infty}\,.\]
The proof can be found in \cite[Chap. 2, Thm 2.2]{these}.

Those commutations imply that the operators $S_{k,\omega}$ and $N_{k,\omega}$ share the same eigenvectors as, respectively, 
$\left[-(\omega \partial_x)^2 - k^2\omega^2\right]$ and $ \left[-(\partial_x \omega)^2 - k^2\omega^2\right]$. The eigenfunctions 
of the operator $\left[ -(\omega \partial_x)^2 - k^2\omega^2\right]$ thus provide us with a diagonal basis for $S_{k,\omega}$. 
They are the solutions of another Sturm-Liouville problem 
\[ (1-x^2) \partial_{xx}y - x \partial_xy - k^2 \omega^2 y = \lambda y\,.\]
Once we set $x = \cos \theta$, $\tilde{y}(\theta) = y(x)$,  $q = \frac{k^2}{4}$, $a = \lambda + 2q$, $\tilde{y}$ is a solution of the 
standard Mathieu equation 
\begin{equation}
\label{MatthieuEq}
\tilde{y}'' + (a - 2q \cos(2\theta)) \tilde{y} = 0\,.
\end{equation}
There exists a discrete set of values $a_{2n}(q)$ for which this equation possesses even and $2\pi$ periodic solutions, which are 
known as the Mathieu cosine functions, and usually denoted by $\textup{ce}_n$. Here, we use the notation $\textup{ce}^k_n$ to 
emphasize the dependency in the parameter $k = \sqrt{2q}$ of those functions. The normalization is taken as
\[ \int_{-\pi}^{\pi} \textup{ce}^k_n(\theta)^2 d\theta = \pi.\]
The Mathieu cosine functions are $L^2$ orthogonal:
\[ \int_{-\pi}^{\pi}\textup{ce}^k_n(\theta) \textup{ce}^k_m(\theta) = \pi \delta_{m,n}\]
so that any even $2\pi$ periodic function in $L^2(-\pi,\pi)$ can be expanded along the functions $\textup{ce}_n$, with the coefficients 
obtained by orthonormal projection. Setting 
\[T_{n}^k \isdef \textup{ce}^k_n(\arccos(x)),\]
in analogy to the zero-frequency case, we have
\[\left[-(\omega \partial_x)^2 - k^2\omega^2\right] T_{n}^k = \lambda_{n,k}^2 T_{n}^k.\]
For large $n$, using the general results from the theory of Hill's equations (see e.g. \cite[eqs. (21), (28) and (29)]{NIST:DLMF}), 
we have the following 
asymptotic formula for $\lambda_{n,k}$:
\[ \lambda_{n,k}^2 = n^2 - \frac{k^4}{16n^2} +o \left(n^{-2}\right). \]
The first commutation established in \autoref{commutations} implies that the Mathieu cosine functions are also the eigenfunctions of the 
single-layer operator. (An equivalent statement is given in \cite[Thm 4.2]{betcke2014spectral}, if we allow the degenerate case $\mu = 0$.) 

A similar analysis can be applied to the hypersingular operator. The eigenfunctions of $\left[-(\partial_x \omega)^2 - k^2 \omega^2\right]$ 
are given by 
\[U_n^k \isdef \frac{\textup{se}_n^k(\arccos(x))}{\omega(x)}\]
where $\textup{se}_n^k$ are the so-called Mathieu sine functions, which also satisfy the Mathieu differential equation \eqref{MatthieuEq}, 
but with the condition that they are $2\pi$ periodic and odd functions. 

One could furthermore expect that the operators \[P_k = \left[-(\omega \partial_x)^2 -k^2 \omega^2\right]^{1/2} \quad \textup{and} \quad Q_k = \left[-(\partial_x\omega)^2 - k^2\omega^2\right]^{-1/2}\]
provide compact perturbations of the inverses of $S_{k,\omega}$ and $N_{k,\omega}$ respectively, and the knowledge of the eigenvalues of $S_{k,\omega}$ and $N_{k,\omega}$ would allow for a simple proof of this fact. Those eigenvalues are not known but the result can be obtained by a more involved analysis. To state precisely the result, let us introduce the following terminology:
\begin{Def}
	A linear operator $A : T^{-\infty} \to T^{-\infty}$ (resp.  $U^{-\infty} \to U^{-\infty}$) is of order $\alpha$ in the scale $(T^s)_s$ (resp. $(U^s)_s$) if for all real $s$, $A$ maps continuously $T^{s}$ to $T^{s-\alpha}$ (resp. $U^s$ to $U^{s-\alpha}$).
\end{Def}
\begin{theorem}[{\cite[Thms. 4 and 6]{averseng}}]
	\label{theorem4}
	The operators
	\[P_k = \left[-(\omega \partial_x)^2 -k^2 \omega^2\right]^{1/2} \quad \textup{and} \quad Q_k = \left[-(\partial_x\omega)^2 - k^2\omega^2\right]^{1/2}\]
	are well defined and satisfy
	\begin{equation}
	P_k S_{k,\omega} = S_{k,\omega} P_k = \frac{I_d}{2} + K_1, \quad \textup{and} \quad N_{k,\omega} = Q_k + K_2
	\label{compactPerturbs}
	\end{equation}
	respectively in $T^{-\infty}$ and $U^{-\infty}$, where $I_d$ is the identity operator and $K_1$ and $K_2$ are of order $-4$ and $-3$ respectively in the scales $T^s$ and $U^s$.
\end{theorem}
Since $K_1$ and $K_2$ are of negative order, they are in particular compact endomorphisms of $L^2_\frac{1}{\omega}$ and $L^2_\omega$ respectively, due to the compact inclusions that hold for 
\[T^{s'} \subset T^s\,, \quad U^{s'} \subset U^s, \quad s' > s\,.\] Therefore, the operators appearing in \eqref{compactPerturbs} are of second-kind in those spaces. The extent to which the dependence in $k$ in $P_k$ and $Q_k$ is optimal is reflected by the next theorem.
\begin{theorem}[{\cite[Cor. 10 ]{averseng}}]
	\label{theocpct}
	Let $K$ be an operator of order $0$ in the scale $T^s$ and let $\Delta_K$ be defined by 
	\[\Delta_K = [-(\omega\partial_x)^2 + K]S_{k,\omega}^2 - \frac{I_d}{4}\,.\]
	Then $\Delta_k$ is of order $-2$ in the scale $T^s$. Moreover, $\Delta_K$ is of order $-4$ if and only if 
	\[K = -k^2 \omega^2 + L\]
	where $L$ is an operator of order $-2$ in the scale $T^s$. 
\end{theorem}
An analogous result holds for the hypersingular operator, see \cite[Cor. 12]{averseng}
Taking $K = 0$, we see that $P_0$ and $Q_0$ are also compact equivalent inverses of $S_{k,\omega}$ and $N_{k,\omega}$, however up to a less regularizing remainder than $P_k$ and $Q_k$. A clear link between the order of the remainder and the numerical performance of the preconditioner remains to be elucidated. We conjecture that a smoother remainder leads to better performances and especially, more robustness with respect to the parameter $k$. This is strongly supported by our numerical results exposed in section 6. 

We note that all the previous (except the commutations) carries over to the more general case of a $C^{\infty}$ non-intersecting open curve $\Gamma$. In this case, we define a weight $\omega_\Gamma$ on the curve by replacing $\partial_x$ by $\omega_\Gamma(r(t)) = \frac{\abs{\Gamma}}{2} \omega(t)$, where $\abs{\Gamma}$ is the length of the curve and $r: [-1,1] \to \Gamma$ is such that for all $x$, $\abs{\partial_x r(x)} = \frac{\abs{\Gamma}}{2}$. Then, the derivation operator $\partial_x$ is replaced by $\partial_\tau$ the tangential derivative on $\Gamma$, and the weighted layer potentials by 
$S_{k,\omega_{\Gamma}} \isdef S_k \frac{1}{\omega_\Gamma}$, and $N_{k,\omega_{\Gamma}} \isdef N_k \omega_\Gamma$,

The previous theoretical analysis suggests to use $P_k$ and $Q_k$ as operator preconditioners for $S_{k,\omega}$ and $N_{k,\omega}$. The remainder of this paper is dedicated to testing this idea in practice. 

\section{Galerkin method}
\label{sec:numerMeth}

\label{subsec:GalerkinSetting}
It is known that the naive piecewise polynomial Galerkin discretization of the non-weighted integral equations with a uniform mesh converges very slowly in terms of the mesh size (the error in energy norm converges in $O(\sqrt{h})$, see e.g. \cite[Thm 1.2]{postell1990h}). Several alternative discretization schemes for the integral equations on open arcs have been proposed in the literature, including mesh grading \cite{postell1990h,costabel1988convergence}, Galerkin method with special singular functions \cite{stephan1984augmented,costabel1987improved}, cosine change of variables \cite{yan1990cosine}, and high-order Nystr\"{o}m methods \cite{bruno2012second}. 

Here, we describe a simple Galerkin setting, suited to the spaces $T^s$ and $U^s$. We use standard piecewise linear functions defined on 
a non-uniform mesh, which is refined towards the edges as follows. Let $X : [0, \abs{\Gamma}]$ be the parametrization of $\Gamma$ by the arclength, where $\abs{\Gamma}$ is the length of the curve. We choose the breakpoints $(X_i)_{1 \leq i \leq N}$ as 
\[X_i = X(s_i)\] 
where $(s_i)_{1 \leq i \leq N}$ are such that the value
\[h_i := \int_{s_{i}}^{s_{i+1}} \frac{ds}{\omega(s)} \]
be (approximately) constant in $i$. 

Such a mesh turns out to be analogous to an algebraically graded mesh with a grading parameter $\beta = 2$. That is to say, near an edge of the curve, the width of the $i-th$ interval is approximately $(ih)^2$. Notice that this modification alone, i.e. using the h-BEM method with a polynomial order $p=1$, is not sufficient to get an optimal rate of convergence.
Indeed, it only leads to a convergence rate in $O(h)$ for the energy norm 
(cf. \cite[Theorem 1.3]{postell1990h}) instead of the expected $O(h^{5/2})$ behavior (to reach such an order of convergence would require $\beta = 5$).

The key ingredient to recover optimal convergence, beside the graded mesh, is to use a weighted $L^2$ scalar product 
(with weight  $\frac{1}{\omega}$ or $\omega$ depending on the considered equation), in order to assemble the operators in their natural spaces. 
We state here the orders of convergence that one gets with this new method, and refer the reader to \cite[Chap. 3, Sec. 2.3]{these} for the proofs. To keep the exposition simple, we also restrict our presentation to the case where $\Gamma = [-1,1]\times\{0\}$ and $k = 0$. 

\paragraph{Dirichlet problem.} For the resolution of the single-layer equation \eqref{Slambda} we use a variational formulation of \eqref{Somegaalpha} 
to compute an approximation $\alpha_h$ of $\alpha$. Namely, let $V_h$ the Galerkin space of (discontinuous) piecewise affine functions 
defined on the mesh $(x_i)_{0\leq i \leq N}$ defined above, and $\alpha_h$ the unique solution in $V_h$ to
\[ \inner{S_{0,\omega} \alpha_h}{\alpha_h'}_\frac{1}{\omega} = \inner{u_D}{\alpha_h'}_\frac{1}{\omega}, \quad \forall \alpha_h' \in V_h\,.\]
We then compute $\lambda_h = \frac{\alpha_h}{\omega}$. Using the notation $C$ to denote any constant that does not depend on the parameter $h$, we then have
\begin{theorem}[{\cite[Thm. 3.1]{these}}]
	If the data $u_D$ is in $T^{s+1}$ for some $-1/2 \leq s \leq 2$, then there holds:
	\[ \norm{\lambda - \lambda_h}_{\tilde{H}^{-1/2}} \leq C h^{s+1/2} \norm{\omega \lambda}_{T^{s}}\leq C h^{s+1/2} \norm{u_D}_{T^{s+1}}.\]
	\label{theOrdreCVDirichlet}
\end{theorem}
In particular, when $u_D$ is smooth, the solution $\alpha = \omega \lambda$ belongs to $T^{\infty}$, and we get the optimal rate of convergence of the error in  
$O(h^{5/2})$.

\paragraph{Neumann problem.} For the numerical resolution of \eqref{Nmu}, we use a variational form for equation \eqref{Nomegabeta} to compute an approximation $\beta_h$ of $\beta$, and solve 
it using a Galerkin method with continuous piecewise affine functions. Introducing $W_h$ the space of continuous piecewise affine functions on the mesh
defined by the points $(x_i)_{0\leq i\leq N}$, we denote by $\beta_h$ the unique solution in $W_h$ to the variational equation:
\begin{equation}
\inner{N_{0,\omega} \beta_h}{\beta_h'}_{\omega} = \inner{u_N}{\beta_h'}_{\omega}, \quad \forall \beta_h' \in W_h.
\label{NomegaBetaGalerk}
\end{equation}
Then, the proposed approximation for $\mu$, given by $\mu_h = \omega \beta_h$, satisfies the following error estimate. 
\begin{theorem}[{\cite[Thm. 3.2]{these}}]
	If $u_N \in U^{s-1}$, for some $\frac{1}{2} \leq s \leq 2$, there holds 
	\[\norm{\mu - \mu_h}_{\tilde{H}^{1/2}} \leq C h^{s - \frac{1}{2}}\norm{\frac{\mu}{\omega}}_{U^{s}}\leq C h^{s - \frac{1}{2}}\norm{u_N}_{U^{s-1}}.\]
	\label{theOrdreCVNeumann}
\end{theorem}

\paragraph{Numerical validation.}
In fact, estimates in the local weighted $L^2$ norms can be derived from the previous, namely:
\[\forall s \in \left[0,2\right], \quad \norm{\alpha - \alpha_h}_{\frac{1}{\omega}} \leq C h^{s}\norm{\alpha}_{T^s}\,\] 
for the Dirichlet problem and 
\[\forall s \in \left[0,2\right], \quad \norm{\beta - \beta_h}_{{\omega}} \leq C h^{s}\norm{\beta}_{U^s}\,\] 
\[\forall s \in \left[1,2\right], \quad \norm{\beta - \beta_h}_{U^1} \leq C h^{s-1}\norm{\beta}_{U^s}\,\] 
for the Neumann problem. 
We verify those rates numerically. For the Dirichlet problem, we solve two test cases $S_{0,\omega} \alpha_1 = u_1$ and $S_{0,\omega} \alpha_2 = u_2$ having the explicit solutions $\alpha_1(x) = \omega(x)$ and $\alpha_2 = \omega(x)^3$, for adequately chosen right hand sides (rhs) $u_1$ and $u_2$. One can check that $\alpha_1 \in T^{s}$ for $s < \frac{3}{2}$ and $\alpha_1 \notin T^{3/2}$, while $\alpha_2 \in T^2$. The $L^2_\frac{1}{\omega}$ error is plotted in \autoref{fig:errL2dir} in each case as a function of the mesh size $h$. We find that the expected rates $O(h^{3/2})$ and $O(h^2)$ predicted by the theory are precisely observed in practice.  

Similarly, for the Neumann case, we solve a a test case $N_{0,\omega} \beta = u_N$ where the solution $\beta$ is explicit. We take $u_N = U_2$ the second Chebyshev polynomial of the second kind. The corresponding solution $\beta$ is proportional to $U_2$ and thus belongs to $U^\infty$. The theory therefore predicts a convergence rate of the error in the $L^2_\omega$ and $U^1$ norms respectively in $O(h^2)$ and $O(h)$. This behavior is again confirmed by our numerical results, exposed in \autoref{fig:errL2Neu}.
\begin{figure}[t]
	\centering
	\includegraphics[scale = 0.4]{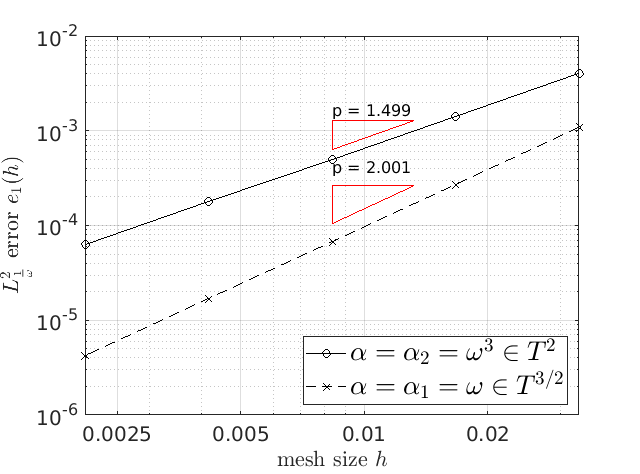}
	\caption{Effective order of convergence of the approximation of the solution $\beta$ to \eqref{Nomegabeta} by the weighted Galerkin method. Two cases are considered where $\alpha \in T^s$ for all $s < 3/2$ but $\alpha \notin T^{3/2}$ (solid line and circles) and $\alpha \in T^2$ (dashed line, crosses) respectively. The approximate slope $p$ is displayed above each curve. Theoretical convergence rates, respectively $O(h^{3/2})$ and $O(h^2)$ are recovered in practice.} 
	\label{fig:errL2dir}
\end{figure}
\begin{figure}[t]
	\centering
	\hspace{-1cm}\includegraphics[scale = 0.4]{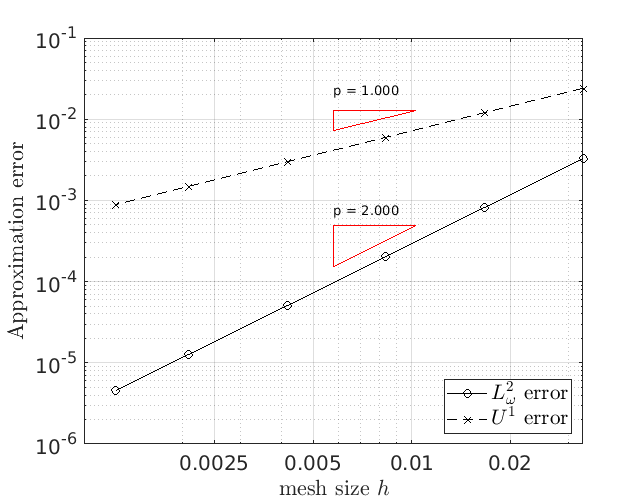}
	\caption{Effective order of convergence of the approximation of the solution $\beta$ to \eqref{Nomegabeta} by the weighted Galerkin method. In this test, the solution $\beta$ lies in $U^2$ and we measure the error in two different norms, respectively $L^2_\omega$ (solid line, circles) and $U^1$ (dashed line, crosses). The approximate slope $p$ is displayed above each curve. The theoretical order of convergence, in $O(h^2)$ and $O(h)$ respectively for the $L^2_\frac{1}{\omega}$ and $U^1$ norms, are recovered in practice. } 
	\label{fig:errL2Neu}
\end{figure}

\section{Building the preconditioners}
Let $X_h$ the considered finite element space ($X_h = V_h$ or $W_h$), and $(\phi_i)_{i}$ the basis functions. For an operator $A$, we denote by $\left[A\right]_p$ 
the Galerkin matrix of the operator for the relevant weight $p(x) = \frac{1}{\omega(x)}$ or $\omega(x)$, defined by
$$
[A]_{p,ij} = \int_\Gamma (A\phi_j)(x) \phi_i(x) p(x)\,dx\,
$$
where $\phi_i$ and $\phi_j$ are the basis functions of the Galerkin space. When the operator $BA$ is a compact perturbation of the identity (either in $T^s$ or $U^s$) then, following \cite{hiptmair2006operator,steinbach1998construction}, we 
precondition the linear system $\left[A\right]_p x = b$ by the matrix $\left[I_d\right]^{-1}_p \left[B\right]_p \left[I_d\right]_p^{-1}$, which amounts to solve
$$
\left[I_d\right]^{-1}_p \left[B\right]_p \left[I_d\right]_p^{-1} [A]_p x = \left[I_d\right]^{-1}_p \left[B\right]_p \left[I_d\right]_p^{-1} b\,.
$$
When $B$ is the inverse of a local operator $C$, then it may be more convenient to compute $\left[C\right]_p$, and solve instead
$$
\left[C\right]_p^{-1}[A]_p x = \left[C\right]_p^{-1}b\,.
$$

The operators 
\[P_k = \left[-(\omega \partial_x)^2 -k^2\omega^2\right]^{1/2} \sim S_{k,\omega}^{-1}, \quad Q_k = \left[-(\partial_x \omega)^2 -k^2\omega^2\right]^{-1/2} \sim N_{k,\omega}^{-1}\]
introduced in section 2 for $k = 0$ (\autoref{theorem1} and \autoref{theorem2}) and section $3$ for $k > 0$ are at the base of our preconditioning strategy, as $P_kS_{k,\omega}$ and $Q_kN_{k,\omega}$ are compact perturbations of the identity (\autoref{theorem4}). To define preconditioners for the linear systems, following the above remark, we need to compute the Galerkin matrices of those operators. 
For $Q_k$, we rewrite 
\[Q_k = \left[-(\partial_x \omega)^2 - k^2 \omega^2 \right]^{-1} \left[-(\partial_x \omega)^2 - k^2 \omega^2 \right]^{1/2}\,.\]
This brings us back to computing the Galerkin matrix of the square root of a differential operator. When the frequency is $0$, we use the method exposed in \cite{hale2008computing}, relying on the discretization of contour integrals in the complex plane. When the frequency is non-zero, the previous method fails since the spectrum of the matrix contains negative values. We then follow Antoine and Darbas \cite{antoine2007generalized} by using a Padé approximation of the square root with regularization and rotation of the branch cut. Let us reproduce here some details of the method, for the reader's convenience. Consider the classical Padé approximation
\[\forall z \in \mathbb{C}, \quad \sqrt{1+ z} \approx R_{N_p}(z) \isdef c_0 + \sum_{j = 0}^{N_p} \frac{a_j z}{1 + b_j z}\]
where the coefficients $c_0$, $a_i$ and $b_i$ are given by the formulas
\[c_0 = 1, \quad a_j = \frac{2}{2N_p + 1}\sin^2\left( \frac{j \pi}{2N_p + 1}\right)\, \quad b_j = \cos^2\left( \frac{j \pi}{2N_p + 1}\right)\]  
It is preferable to use a ``rotated" version of this approximation to avoid the singularity related to the branch cut for $X < -1$:
\[\sqrt{1+ z} = e^{i\frac{\theta}{2}} \sqrt{(1 + z)e^{-i\theta}} \approx e^{i\frac{\theta}{2}} R_{N_p}\left(1 + \left[(1 + z)e^{-i\theta} - 1\right]\right)\]
This yields the new approximation
\[\sqrt{1+ z} \approx C_0 + \sum_{i = 0}^{N_p} \frac{A_i z}{1 + B_iz}\]
where 
\[C_0 = e^{\frac{i \theta}{2}} R_{N_p}(e^{-i\theta} - 1)\,,\]
\[A_j = \frac{e^{-j \theta}a_j}{(1 + b_j(e^{-i\theta} - 1))^2}\,, \quad B_j = \frac{e^{-i\theta}b_j}{1 + b_j(e^{-i\theta} - 1)}\,.\]
This provides a good approximation of the square root in any region of the real line away from $X = -1$, as described by the next result.
\begin{Lem}
	\label{lemPrecisionsPade}
	Let $\theta \in \R$ and let $z$ a complex number. Let $r = \abs{z+ 1}$. One has 
	\[\abs{\sqrt{1+z} - C_0 - \sum_{j = 0}^{N_p} \frac{A_j z}{1 + B_j z}} \leq 2 \sqrt{r} \abs{\gamma(r,\theta)}^{2N_p + 1}\]
	where 
	\[\gamma(r,\theta) = \frac{\sqrt{r}e^{i\frac{\theta}{2}}-1}{\sqrt{r}e^{i\frac{\theta}{2}}+1}\]
\end{Lem}
As a consequence, it is not difficult to check that, if $\theta \in (-\pi,\pi)$, then when $N_p \to \infty$, the Padé approximants converge exponentially in the uniform error in any region of the form $\delta < |z + 1| < R$ where $0 < \delta < 1$ and $R > 1$. The previous scheme can be exploited to approximate an operator $\sqrt{X - k^2 I_d}$ where $X$ is a positive self-adjoint operator. Writing
\[\sqrt{X - k^2I_d} = ik \sqrt{I_d + \frac{-X}{k^2}}\]
we have
\[\sqrt{X - k^2I_d} \approx ik \left(C_0 + \sum_{j = 0}^{N_p}A_jX(B_jX - k^2)^{-1}\right)\]
Using \autoref{lemPrecisionsPade}, one can conclude that, when $N_p$ is large enough, this yields a good approximation in the eigenspaces that are associated to eigenvalues $\lambda$ such that
\begin{equation}
\label{equationsSurR}
\frac{\lambda}{k^2} \in [-R,R] \setminus [1 - \delta, 1 + \delta]\,.
\end{equation}
In our context, the eigenvalues $\lambda \approx k^2$ correspond to the so-called ``grazing modes". To deal with them, Marion Darbas introduced in her thesis \cite{darbas2004preconditionneurs} a regularization recipe, which consists in adding to the wavenumber some damping. Namely, letting $\varepsilon > 0$, the approximation is replaced by 
\[\sqrt{X - k^2 I_d} \approx ik \left(C_0 I_d + \sum_{j = 0}^{N_p} A_j X \left(B_j X - (k + i\varepsilon)^2I_d\right)^{-1}\right)\]
Based on those considerations, we can approximate the Galerkin matrix of the operator appearing on the left by 
\begin{equation}
\label{DefPrecPade}
\left[\sqrt{X - k^2 I_d}\right]_p \approx ik \left(C_0 [I_d]_p + \sum_{j = 0}^{N_p} A_j [X]_p \left(B_j [X]_p - (k + i\varepsilon)^2[I_d]_p\right)^{-1}\right)\,.
\end{equation}
When $X$ is a local operator, this formula involves sparse matrix products and sparse linear system resolutions, which can be performed efficiently. To get the Galerkin matrices of $P_k$ and $Q_k$, we apply this strategy with $X = -(\omega \partial_x)^2 + k^2(I_d -\omega^2)$ in the case of the Dirichlet problem and $X = -(\partial_x \omega)^2 + k^2(I_d - \omega^2)$ for the Neumann problem. We use the parameters $N_p = 15$, $\theta = \frac{\pi}{3}$ and $\varepsilon = 0.05 k^{1/3}$. Our numerical results do not depend crucially on those choices, see \autoref{secPade}. The choice of $N_p$ is dictated by the ratio $\frac{N_{dof}}{k}$ where $N_{dof}$ is the dimension of the Galerkin space. In our tests, this ratio will be held fixed unless stated otherwise, and thus, a fixed value of $N_p$ yields good results for all $k$. An informal way to explain this fact is that for our choices of operator $X$, the largest eigenvalues of the Galerkin matrix behave as $\lambda_M \leq C_1 N_{dof}^2$. Using a fixed number of points per wavelength gives in turn $N_{dof} = C_2 k$ and thus 
\[\lambda_M \leq C_1 C_2 k^2\,.\]
Choosing $N_p$ such that \eqref{equationsSurR} holds with $R = C_1C_2$, and forgetting about the grazing modes, we thus ensure a small error in \eqref{DefPrecPade} independently of $k$. On the other hand, when $k$ is fixed and the mesh is refined, it is necessary to increase $N_p$ to maintain accuracy. 

\section{Numerical results}
\label{sec:NumericalResutls}

In this section, we present some numerical results concerning the efficiency of the preconditioners defined above for solving the linear systems arising from the Galerkin method detailed in section 4. The linear systems are solved with the GMRES method \cite{saad1986gmres} with no restart and a tolerance of $\varepsilon = 10^{-8}$. Execution times are reported only to show in which case the preconditioned linear system is solved faster than the non-preconditioned system. We report the best timing of three successive runs when the time is less than $5$ seconds. If the number of iterations is greater than 500, we stop the calculations and report the time to reach the 500th iteration. The computations are also interrupted when they last more than $15$ minutes and in this case no iteration number is reported. All the simulations are performed on a personal laptop running on an eight cores intel i7 processor with a clock rate of 2.8GHz. The method is implemented in the language Matlab R2018a. In this section, $k$ stands for the wave number, $N$ for the number of mesh points and $\abs{\Gamma}$ for the length of a curve $\Gamma$. Moreover, $\omega_\Gamma$ is the weight defined on any curve $\Gamma$ by 
\[\omega_\Gamma(r(t)) = \frac{\abs{\Gamma}}{2}\sqrt{1 - t^2}\]
where $r : [-1,1] \to \Gamma$ is a parametrization of $\Gamma$ satisfying
\[\forall t \in [-1,1], \quad  \norm{\partial_t r(t)} = \frac{\abs{\Gamma}}{2}\,.\]
For problems such that $N > 5 \times 10^3$, we use the Efficient Bessel Decomposition \cite{averseng2017} to compress the Galerkin matrices. 

\subsection{Laplace equation on the flat segment}

We start by testing the performance of the exact inverses, characterized in section 2, as preconditioners. This is rather a validation stage. 

\paragraph{Flat segment, Laplace-Dirichlet problem.}   In \autoref{TableNitTimeLaplaceDirichlet}, we report the timings and number of GMRES iterations for the Laplace weighted single-layer equation 
\[S_{0,\omega} \alpha = u_D\,.\]
Two cases are considered, first without any preconditioner, and then with a preconditioner given by the exact inverse $2\sqrt{-(\omega \partial_x)^2} + \frac{2}{\ln(2)} \pi_0$ (see \autoref{theorem1} and the previous section for the detailed construction of the preconditioner). The rhs is chosen as 
\[\forall x \in [-1,1]\,, \quad u_D(x) = \left(x^2 + \frac{1}{N^2}\right)^{-\frac{1}{2}}\,.\]
A graph of the history of the GMRES relative residual is given in \autoref{FigureNitLaplaceDirichlet} for a mesh with $N = 2000$ node points. 

\begin{table}[t]
	\begin{center}
		\begin{tabular}{m{4em}|m{4em}|m{4em}|m{4em}| m{4em}} 
			\hline
			\multicolumn{1}{c|}{ }&
			\multicolumn{2}{c|}{with Prec.}&\multicolumn{2}{c}{without Prec.}\\
			\hline
			$N$ & $n_{it}$& t(s) & $n_{it}$ & t(s)\\
			\hline\hline
			500 & 8 & $< 0.1$ & 79 & $< 0.1$ \\
			\hline
			2000 & 8 & $<0.1$ & 128 & 0.3 \\
			\hline
			8000 & 7 & $0.5$ & 218 & 11.5 \\
			\hline
			32000 & 8 & $2.2$& 347 & 89 \\
			\hline
		\end{tabular}	
		\caption{Computing time and number of GMRES iterations for the numerical resolution of the weighted Laplace single-layer integral equation on the segment respectively with the square root preconditioner and without preconditioner.}
		\label{TableNitTimeLaplaceDirichlet}
	\end{center}
\end{table}

\paragraph{Flat segment, Laplace-Neumann problem.} 
For the Laplace weighted hypersingular equation 
\[N_{0,\omega}\beta = u_N\,,\]
we also report in \autoref{TableNitTimeLaplaceNeumann} the timings and number of iterations of the GMRES method first without preconditioner, and then with the preconditioner obtained from the operator $[-(\partial_x\omega)^2]^{-1/2}$ by the method described in the previous section. The rhs is chosen as 
\[\forall x \in [-1,1], \quad  u_N(x) = \left(x^2 + \frac{1}{N^2}\right)^{1/2}\,.\]
A graph of the history of the GMRES relative residual is given in \autoref{FigureNitLaplaceDirichlet} for a mesh mesh with $N = 2000$ node points. 

\begin{table}[H]
	\begin{center}
		\begin{tabular}{m{4em} | m{4em} | m{4em} | m{4em} | m{4em}} 
			\hline
			\multicolumn{1}{c|}{ }&
			\multicolumn{2}{c|}{with Prec.}&\multicolumn{2}{c}{without Prec.}\\
			\hline
			$N$ & $n_{it}$& t(s) & $n_{it}$ & t(s)\\
			\hline\hline
			500 & 5 & $< 0.1$ & 333 & 0.3\\
			\hline
			2000 & 5 & $<0.1$ & $>500$ & 2\\
			\hline
			8000 & 7 & 1.1 & $>500$ & 60 \\
			\hline
			32000 & 6 & 4 & $>500$ & 725\\
			\hline
		\end{tabular}
	\end{center}
	\caption{Computing time and number of GMRES iterations for the numerical resolution of the weighted Laplace hypersingular integral equation on the segment respectively with the square root preconditioner and without preconditioner.}
	\label{TableNitTimeLaplaceNeumann}
\end{table}
We observe in both Dirichlet and Neumann cases a low and stable number of iteration which is the expected behavior. In the Neumann problem, the presence of the preconditioner leads to huge speedups.

\begin{figure}[H]
	\centering
	\begin{subfigure}[]{0.45\linewidth}
		\centering
		\includegraphics[width=\linewidth]{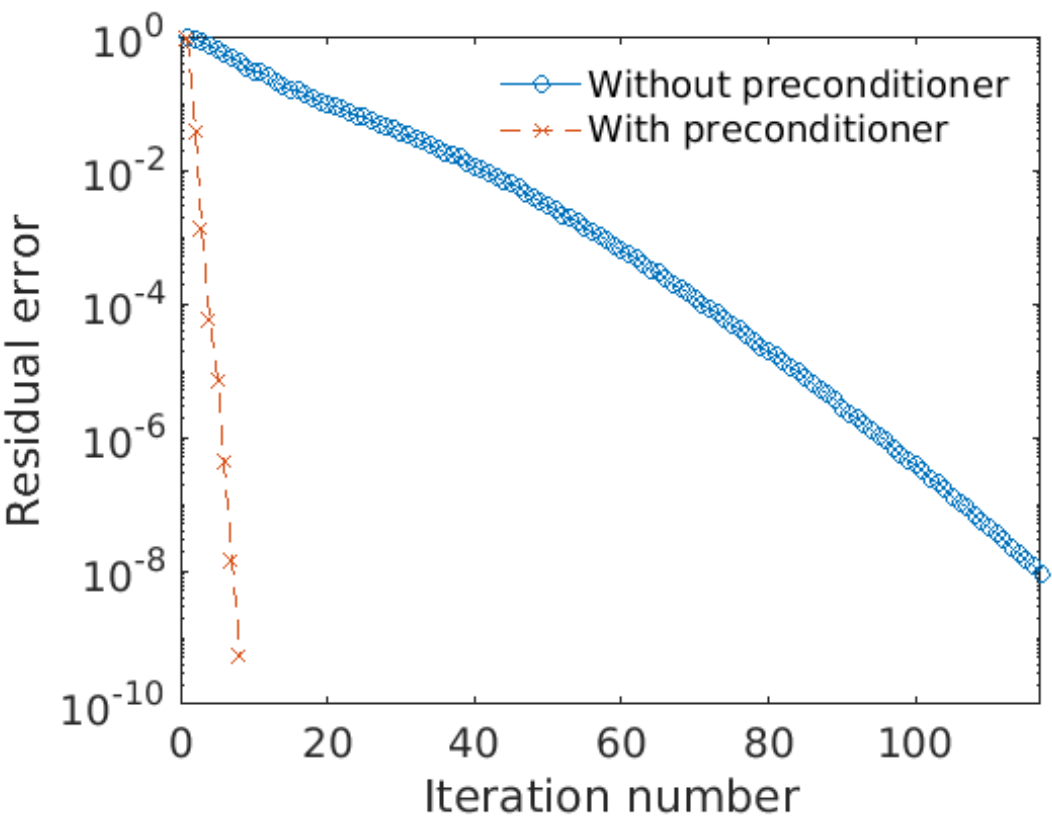}
		\caption{Laplace weighted single-layer}
	\end{subfigure}
	\begin{subfigure}[]{0.45\linewidth}
		\centering
		\includegraphics[width=\linewidth]{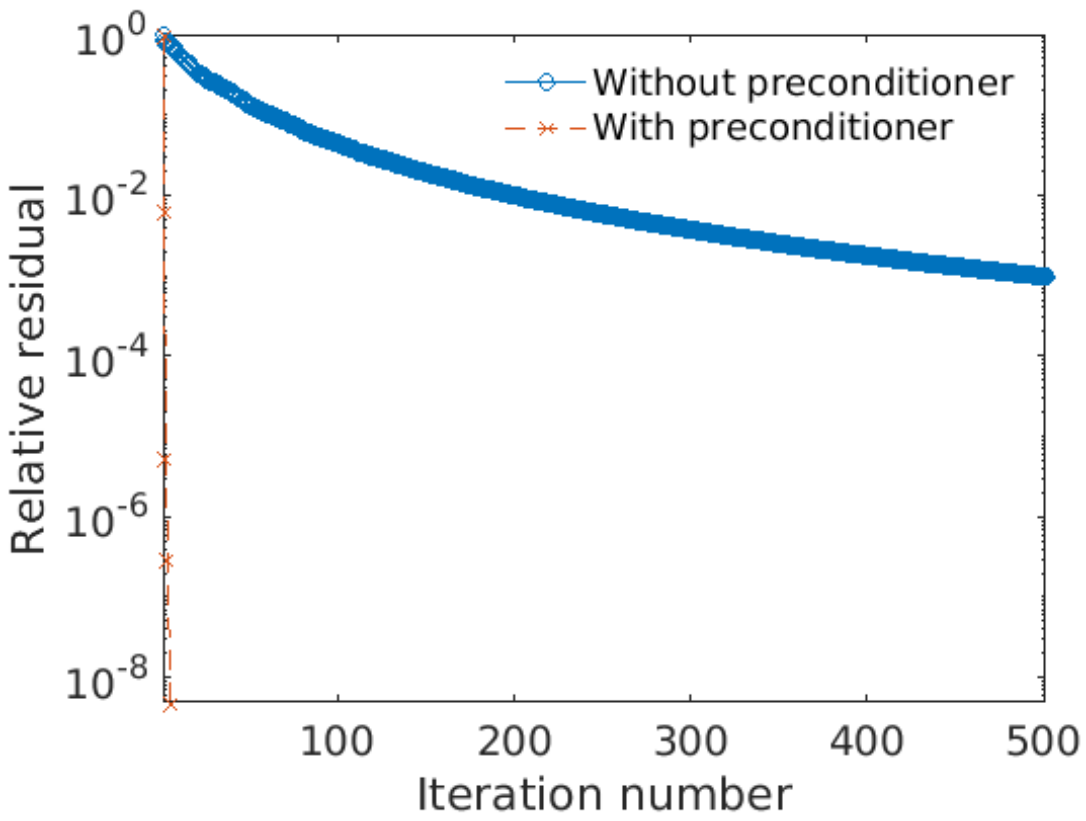}
		\caption{Laplace weighted hypersingular}
	\end{subfigure}
	\caption{Comparison of the hisotry of the GMRES relative residuals for the resolution of the Laplace ($k = 0$) weighted singley-layer (left) and hypersingular (right) integral equations on the flat segment with a mesh of size $N = 2000$, respectively without preconditioner (blue circles) and with the square root preconditioner (red crosses). In this simple case, the preconditioners are based on the exact inverses of the weighted layer potentials}
	\label{FigureNitLaplaceDirichlet}
\end{figure}

\subsection{Helmholtz equation on the flat segment}

We now turn our attention to the Helmholtz equation ($k> 0$). From now on, unless stated otherwise, the number of segments in the discretization is set to $N \approx 5k \abs{\Gamma}$ (5 points per wavelength). 

\paragraph{Flat segment, Helmholtz-Dirichlet problem.} 
In \autoref{TableNitTimeHemholtzDirichlet} we report the number of GMRES iterations for the numerical resolution of the weighted single-layer integral equation 
\[S_{k,\omega}\alpha = u_D\]
on the flat segment $\Gamma = [-1,1]\times \{0\}$, when the using a preconditioner based on the opretor 
$$\sqrt{-(\omega \partial_x)^2 - k^2 \omega^2},$$ 
as compared to the case where no preconditioner is used.  We take, for the Dirichlet data, the plane wave $u_D(x) = e^{ikx}$. We also provide, in 
\autoref{DirichletHelmholtzSeg}, the history of the GMRES relative residual with and without preconditioner, for a problem with $k\abs{\Gamma}=200\pi$. When the preconditioner is used, the number of iterations is approximately reduced by a factor $10$. 
\begin{table}
	\begin{center}
		\begin{tabular}{m{4em} | m{4em} | m{4em} | m{4em} | m{4em}} 
			\hline
			\multicolumn{1}{c|}{ }&
			\multicolumn{2}{c|}{with Prec.}&\multicolumn{2}{c}{without Prec.}\\
			\hline
			${k\abs{\Gamma}}$ & $n_{it}$& t(s) & $n_{it}$ & t(s)\\
			\hline\hline
			50$\pi$ & 8 & $<0.1$ & 88 & 0.2\\
			\hline
			200$\pi$  & 10 & 0.2 & 123 & 1.4\\
			\hline
			400$\pi$  & 13 & 4 & 145 & 40\\
			\hline
			800$\pi$  & 16 & 15 & 155 & 110 \\
			\hline
			1600$\pi$ & 20 & 70 & 199 & 642\\
			\hline
		\end{tabular}
	\end{center}
	\caption{Computing time and number of GMRES iterations for the numerical resolution of the Helmholtz weighted single-layer integral equation on the segment respectively with the square root preconditioner and without preconditioner.}
	\label{TableNitTimeHemholtzDirichlet}
\end{table}
\paragraph{Flat segment, Helmholtz-Neumann problem.} 
We run the same numerical comparisons, this time for the Helmholtz weighted hypersingular integral equation 
\[N_{k,\omega}\mu = u_N\]
on $\Gamma = [-1,1] \times \{0\}$ and taking the preconditioner based on the operator 
\[\left[-(\partial_x\omega)^2 -k^2 \omega^2\right]^{-1/2}\,.\]
Results are given in \autoref{TableNitTimeHemholtzNeumann} for different meshes and in
\autoref{NeumannHelmholtzSeg} for the history of the GMRES relative residual in a case where $k \abs{\Gamma} = 200\pi$. The rhs is chosen as the normal derivative of a diagonal plane wave 
\[u_{inc}(x,y) = e^{i\frac{k \sqrt{2}}{2}(x+y)}\,.\] 
Huge differences, both in time and number of iterations are
shown in favor of the preconditioned system.
\begin{table}[H]
	\begin{center}
		\begin{tabular}{m{4em} | m{4em} | m{4em} | m{4em} | m{4em}} 
			\hline
			\multicolumn{1}{c|}{ }&
			\multicolumn{2}{c|}{with Prec.}&\multicolumn{2}{c}{without Prec.}\\
			\hline
			${k\abs{\Gamma}}$ & $n_{it}$& t(s) & $n_{it}$ & t(s)\\
			\hline\hline
			50$\pi$ & 10 & $<0.1$ & $>500$ & 3.1\\
			\hline
			200$\pi$  & 13 & 0.3 & $>500$ & 12 \\
			\hline
			400$\pi$  & 14 & 7 & $>500$ & 260\\
			\hline
			800$\pi$  & 18 & 34 & - & $>15$min \\
			\hline
			1600$\pi$ & 25 & 270 & - & $>15$min\\
			\hline
		\end{tabular}
	\end{center}
	\caption{Computing time and number of GMRES iterations for the numerical resolution of the Helmholtz weighted hypersingular integral equation on the segment respectively with the square root preconditioner and without preconditioner.}
	\label{TableNitTimeHemholtzNeumann}
\end{table}

\begin{figure}[t]
	\centering
	\begin{subfigure}[]{0.5\linewidth}
		\centering
		\includegraphics[width=\linewidth]{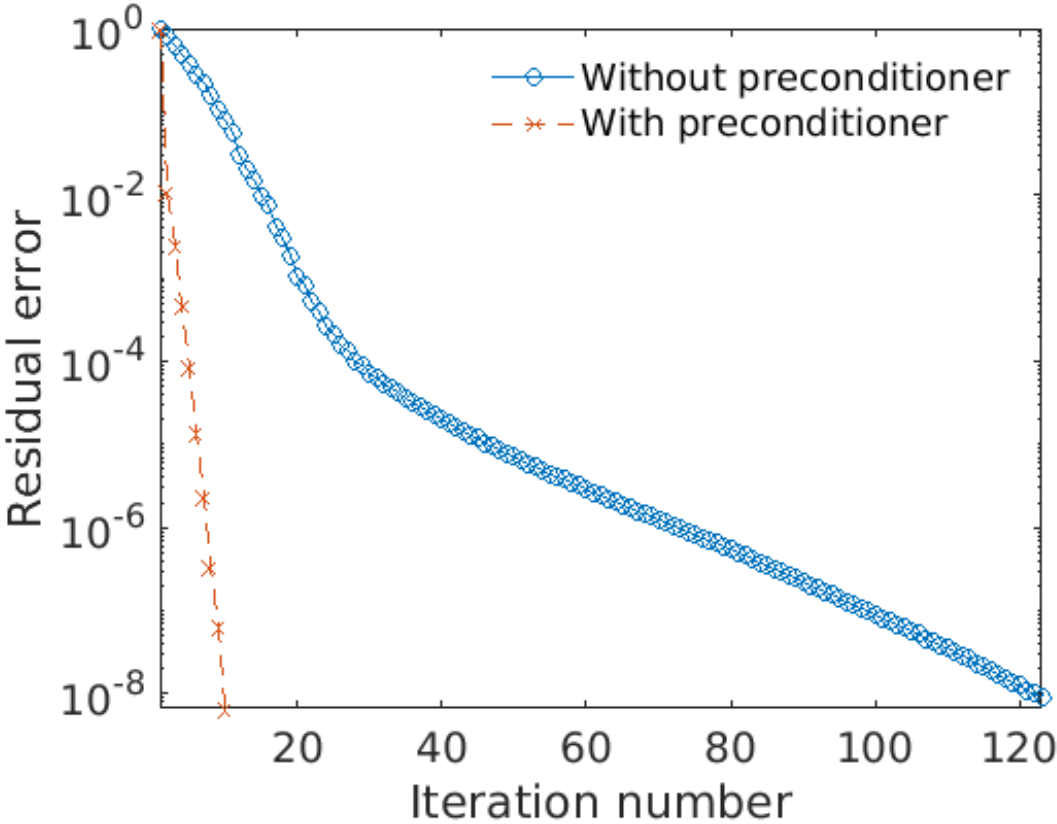}
		\caption{Helmholtz weighted single-layer}
		\label{DirichletHelmholtzSeg}
	\end{subfigure}%
	\begin{subfigure}[]{0.5\linewidth}
		\centering
		\includegraphics[width=\linewidth]{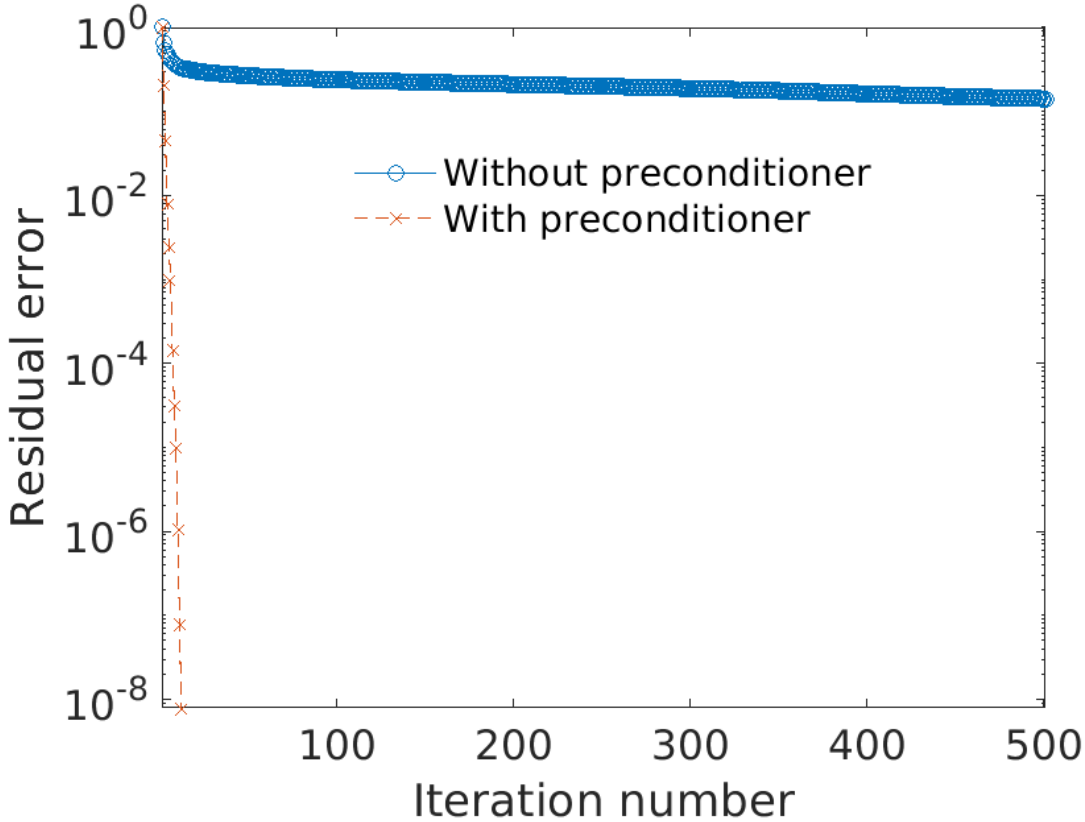}
		\caption{Helmholtz weighted hypersingular}
		\label{NeumannHelmholtzSeg}
	\end{subfigure}
	\caption{Comparison history of the GMRES relative residuals for the resolution of the Helmholtz ($k > 0$) weighted singley-layer (left) and hypersingular (right) integral equations on the flat segment for $k\abs{\Gamma} = 200\pi$, $N \approx 3500$, respectively without preconditioner (blue circles) and with the square root preconditioner (red crosses)}
	\label{FigureNitHelmDirichlet}
\end{figure}

\subsection{Helmholtz equation on non-flat arcs} 

Recall that for any non-flat arc $\Gamma$, the Helmholtz weighted single- and hypersingular layer potentials are defined by 
\[S_{k,\omega_\Gamma}\phi \isdef S_{k} \left(\frac{\phi}{\omega_\Gamma}\right)\,,\quad N_{k,\omega_\Gamma}\phi \isdef N_k (\omega_\Gamma \phi)\,.\]

\paragraph{Spiral-shaped arc} We first consider a spiral-shaped arc of equation 
\[x(t) = e^{0.4(t-0.2)} \cos(2(t - 0.2))\,,\]
\[y(t) = e^{0.4(t-0.2)} \sin(2(t - 0.2))\,,\]
for $t \in [-1,1]$. The curve has a length of about $\abs{\Gamma} = 3.88$. We report in Tables \ref{SpiralDir} and \ref{SpiralNeu} the number of iterations and computing times respectively for the Dirichlet and Neumann problems with the rhs given by
\[u_D = u_{inc|\Gamma}\, \quad u_N = \frac{\partial u_{inc}}{\partial n}\]
where for all $(x,y)$, $u_{inc}(x,y) = e^{ikx}$.  To illustrate the problem, the scattering pattern with Neumann boundary conditions is shown in \autoref{ArcOfSpiral} for this geometry. The preconditioning performances are qualitatively similar to the case of the flat segment. This shows that the preconditioning strategy is also efficient in presence of non-zero curvature.
\begin{figure}
	\centering
	\includegraphics[width=0.6\linewidth]{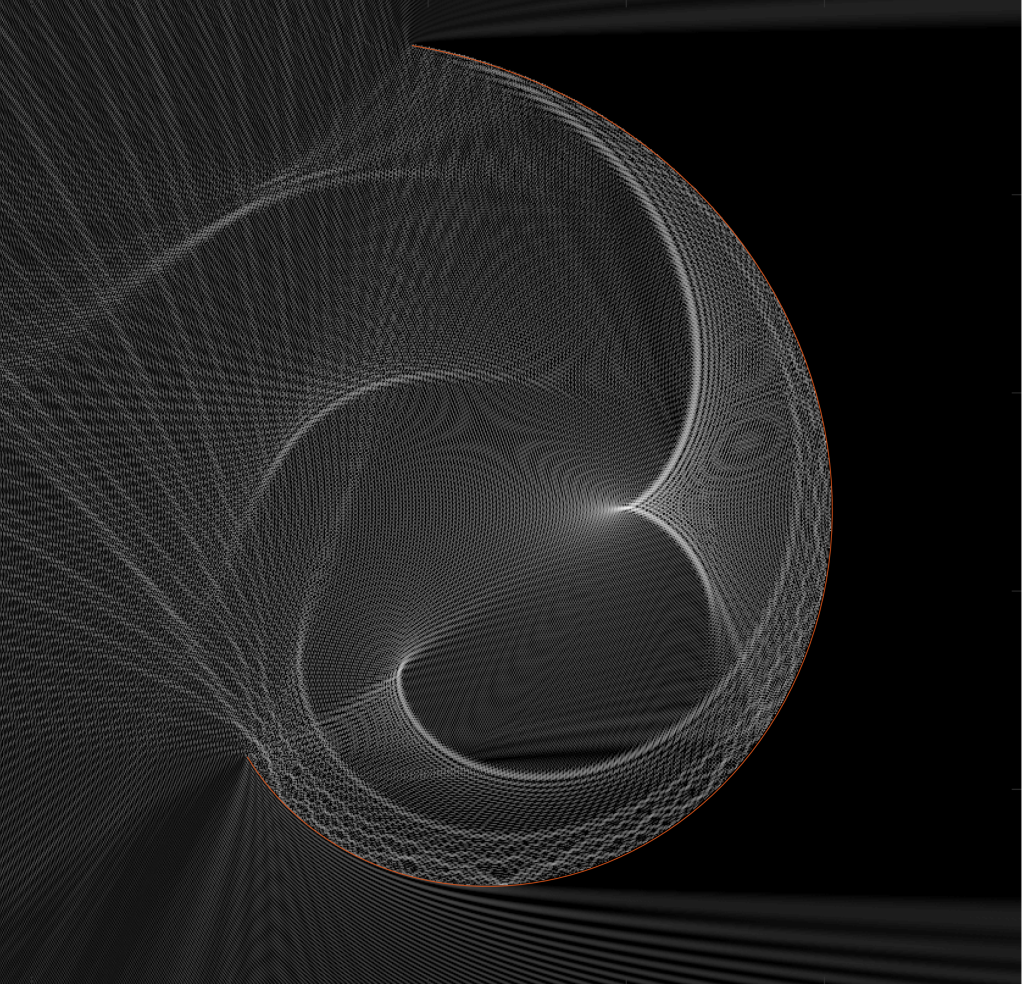}
	\caption{Scattering pattern of a plane wave with left to right incidence
		on the spiral-shaped sound-hard screen (Neumann boundary conditions), with $k\abs{\Gamma} = 800\pi$. After the resolution of the integral equation, the computation of the image is accelerated by
		the EBD method \cite{averseng2017}.}
	\label{ArcOfSpiral}
\end{figure}
\vspace{-0.5cm}
\begin{table}
	\begin{center}
		\begin{tabular}{m{4em} | m{4em} | m{4em} | m{4em} | m{4em}} 
			\hline
			\multicolumn{1}{c|}{ }&
			\multicolumn{2}{c|}{With prec.}&\multicolumn{2}{c}{Without prec.}\\
			\hline
			$k\abs{\Gamma}$ & $n_{it}$& t(s) & $n_{it}$ & t(s)\\
			\hline\hline
			$50\pi$ & 19 & 0.1 & 93 & 0.2\\
			\hline
			$200\pi$ & 24 & 0.55 & 136 & 1.77\\
			\hline
			$400\pi$ & 27 & 4.7 & 160 & 30\\
			\hline
			$800\pi$ & 30 & 16.5 & 190 & 92\\
			\hline
			$1600\pi$ & 32 & 71 & 217 & 456\\
			\hline
		\end{tabular}
	\end{center}
	\caption{Computing time and number of GMRES iterations for the numerical resolution of the Helmholtz weighted single-layer integral equation on the spiral-shaped arc respectively with the square-root preconditioner and without preconditioner}
	\label{SpiralDir}
\end{table}
\begin{table}
	\begin{center}
		\begin{tabular}{m{4em} | m{4em} | m{4em} | m{4em} | m{4em}} 
			\hline
			\multicolumn{1}{c|}{ }&
			\multicolumn{2}{c|}{With prec.}&\multicolumn{2}{c}{Without prec.}\\
			\hline
			$k\abs{\Gamma}$ & $n_{it}$& t(s) & $n_{it}$ & t(s)\\
			\hline\hline
			$50\pi$ & 22 & 0.15 & $>500$ & 3\\
			\hline
			$200\pi$ & 31 & 0.7 & $>500$ & 15.4\\
			\hline
			$400\pi$ & 34 & 13 & $>500$ & 209\\
			\hline
			$800\pi$ & 35 & 46 & - & $>15$min\\
			\hline
			$1600\pi$ & 42 & 283 & - &$>15$min\\
			\hline
		\end{tabular}
	\end{center}
	\caption{Computing time and number of GMRES iterations for the numerical resolution of the Helmholtz weighted hypersingular integral equation on the spiral-shaped arc respectively with the square-root preconditioner and without preconditioner}
	\label{SpiralNeu}
\end{table}

\paragraph{Non-smooth arc} We consider now the case where $\Gamma$ is the V-shaped arc given by the parametric equations
\[x(t) = t\sin\frac{\theta}{2}, \quad y(t) = |t|\cos\frac{\theta}{2}, \quad t \in [-1,1]\]
where $\theta$ is a parameter. When $\theta = \pi$, $\Gamma$ is the flat segment. When $\theta \neq \pi$, this arc has a corner in the middle of angle $\theta$. Since it is not a smooth arc, the theory presented in this work does not apply. For example, the solution $\alpha$ to the weighted single-layer integral equation 
\[S_{k,\omega_\Gamma} \alpha = u_D\]
where $u_D$ is a smooth function, has a singularity at the corner of $\Gamma$. This is illustrated in \autoref{singularityCorner} where we plot the solution $\alpha(s)$ as a function of the arclength $s$ for a rhs $u_D = e^{ikx}$, when $\theta = \frac{\pi}{2}$. Despit this singularity, the number of GMRES iterations remains independent of the mesh size for a fixed frequency. This result is reported in \autoref{TableNitVshapedkGammaFixed}, where we compare, for $k\abs{\Gamma} = 50$ fixed, the number of GMRES iterations for the resolution of the Helmholtz weighted single-layer integral equation respectively on the flat segment and on the V-shaped curve, for different values of $N$. 

We show in \autoref{tableNitAngleHelmholtzDirichlet} the influence of the frequency on the preconditioning performances for the Helmholtz weighted single-layer integral equation on the V-shaped arc of angle $\theta = \frac{\pi}{2}$. The results are qualitatively the same as in the case of a smooth curve. To illustrate the problem, the scattering pattern for a sound-hard V-shaped arc of angle $\frac\pi2$ (Neumann conditions) is shown in \autoref{angle} for $k\abs{\Gamma} = 50$.

\begin{figure}
	\centering
	\includegraphics[width=0.6\textwidth]{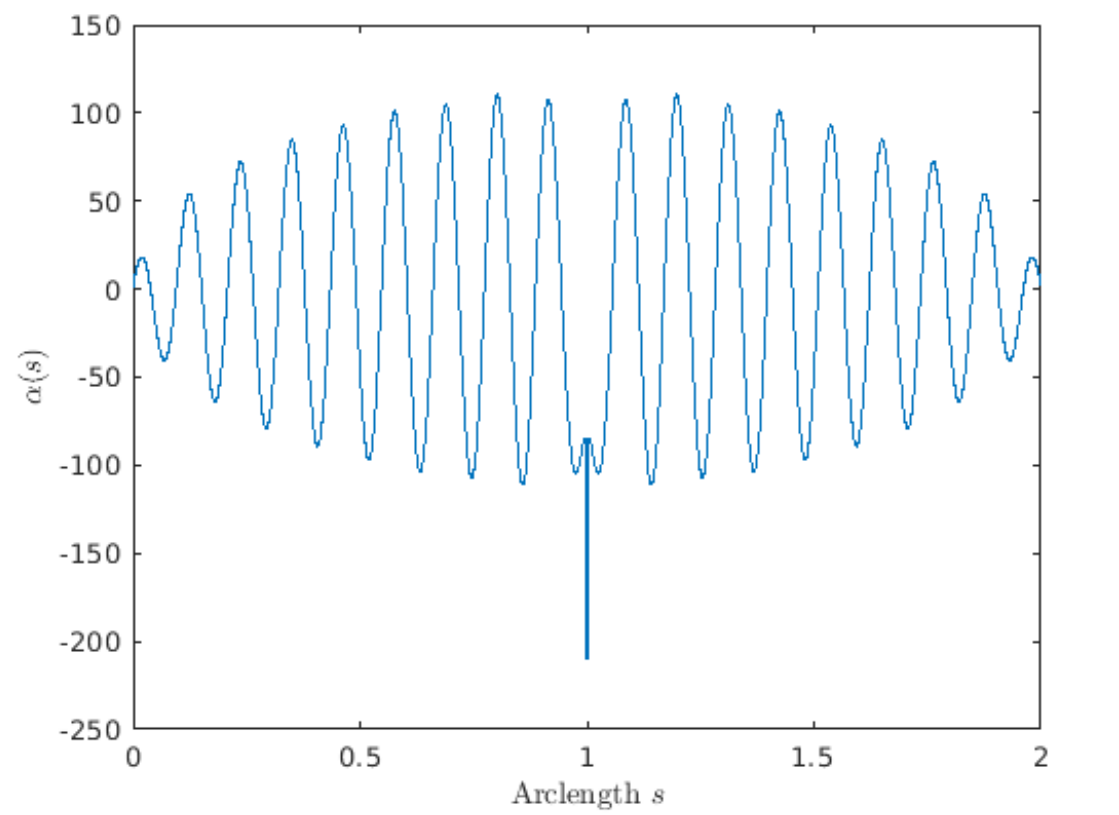}
	\caption{Plot of the solution of the Helmholtz weighted single-layer integral equation $S_{k,\omega_\Gamma} = u_D$ where $\Gamma$ is the V-shaped arc with $\theta = \frac{\pi}{2}$ and $u_D = u_{inc|\Gamma}$ where $u_{inc}(x,y) = e^{iky}$. Notice the singularity at $s = 1$}
	\label{singularityCorner}
\end{figure}
\begin{table}
	\begin{center}
		\begin{tabular}{m{5em} | m{4em} | m{4em} | m{4em} | m{4em}} 
			\hline 
			$N/(k\abs{\Gamma})$ & flat seg. & $\theta = \frac{3\pi}{4}$ & $\theta = \frac{\pi}{2}$ & $\theta = \frac{\pi}{6}$ \\\hline \hline
			2.5 &8 &9 &10 &17\\\hline
			5&7&8&9&17\\\hline
			7.5&7&8&10&17\\\hline
			10&7&8&10&17\\\hline
			12.5&7&8&9&17\\\hline
			15&7&8&10&17\\\hline
		\end{tabular}
	\end{center}
	\caption{Number of GMRES iterations for the numerical resolution of the Helmholtz weighted single-layer integral equation with the square-root preconditioner, from left to right, one the flat segment, and on V-shaped arcs with increasing singularities ($\theta = \frac{3\pi}{4},\frac{\pi}{2},\frac{\pi}{6}$). In all cases, the parameeter $k\abs{\Gamma}$ is fixed to $50$ and the mesh is progressively refined. In this test, we use $N_p = 60$ Padé approximants and $u_{D} = u_{inc|\Gamma}$ where $u_{inc}(x,y) = e^{iky}$.}
	\label{TableNitVshapedkGammaFixed}
\end{table}
\begin{figure}
	\centering
	\begin{subfigure}[b]{0.45\linewidth}
		\centering
		\includegraphics[width=\linewidth]{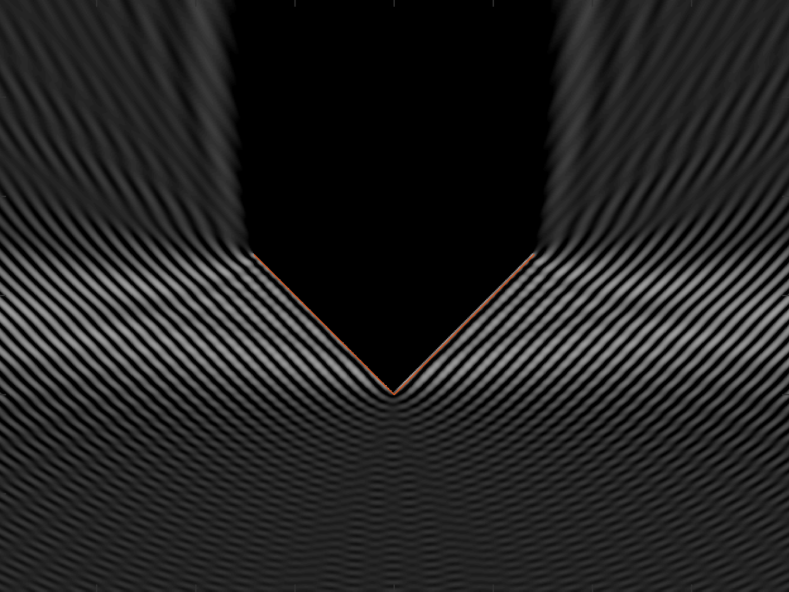}
	\end{subfigure} \hspace{0.1cm}
	\begin{subfigure}[b]{0.45\linewidth}
		\centering
		\includegraphics[width=0.997\linewidth]{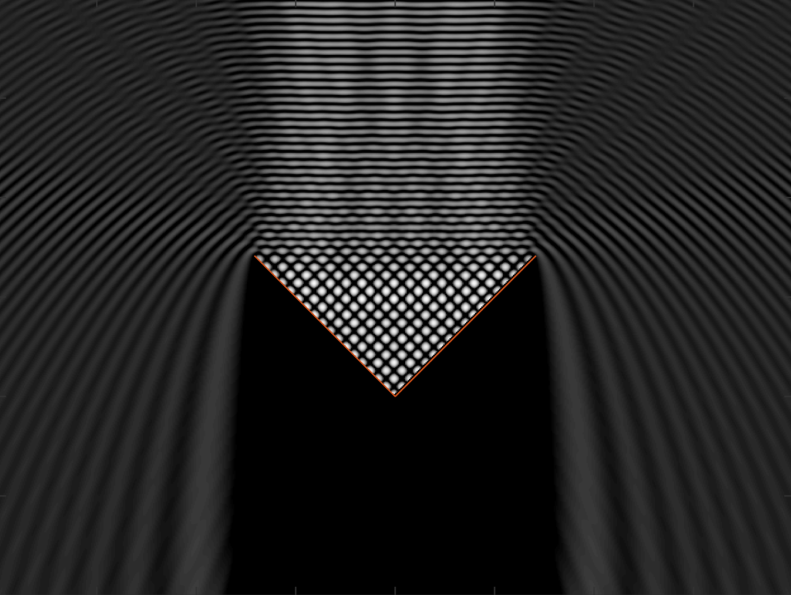}
	\end{subfigure}
	\caption{Scattering patterns for a plane wave with vertical incidence (left: bottom to top, right: top to bottom) for a V-shaped sound-hard (Neumann boundary conditions) screen with $\theta = \frac{\pi}{2}$ and $k\abs{\Gamma}= 50\pi$. On the left, notice that the energy is deflected in the orthogonal directions. On the right, notice the resonating aspect of the solution}
	\label{angle}
\end{figure}

\begin{table}
	\begin{center}
		\begin{tabular}{m{4em} | m{4em} | m{4em} | m{4em} | m{4em}} 
			\hline 
			& \multicolumn{2}{c}{With prec.} & \multicolumn{2}{c}{Without prec.}\\\hline
			$k\abs{\Gamma}$ & $n_{it}$ & $t(s)$ & $n_{it}$ & $t(s)$\\
			\hline\hline 
			50$\pi$ & 9 & $<0.1$ & 97 & $0.25$\\\hline
			200$\pi$ & 10 & 0.3 & 157 & 3.1\\\hline
			400$\pi$& 11 & 3.1 & 190 &  41\\ \hline
			800$\pi$ & 14 & 8 & 231 & 138 \\ \hline
			1600$\pi$ & 18 & 48 & - & $>15$min\\ \hline
		\end{tabular}
	\end{center}
	\caption{Computing time and number of GMRES iterations for the numerical resolution of the Helmholtz weighted single-layer integral equation on the V-shaped arc, with $\theta = \frac{\pi}{2}$ respectively with the square root preconditioner and without preconditioner.}
	\label{tableNitAngleHelmholtzDirichlet}
\end{table}

\subsection{Influence of the number of Padé approximants}
\label{secPade}

The method is not very sensitive to the number of Padé approximantes, i.e. the parameter $N_p$ in \autoref{DefPrecPade}. We show this in \autoref{sensibilitePade} in the case of the Dirichlet problem for the spiral-shaped arc with $k\abs{\Gamma} = 200\pi$ and $u_D = e^{ikx}$. The parameter $\varepsilon$ and the angle of the branch rotation $\theta$ remain fixed (see section 5). 

\begin{figure}[H]
	\centering
	\begin{subfigure}[t]{0.45\linewidth}
		\centering
		\includegraphics[height=4.3cm]{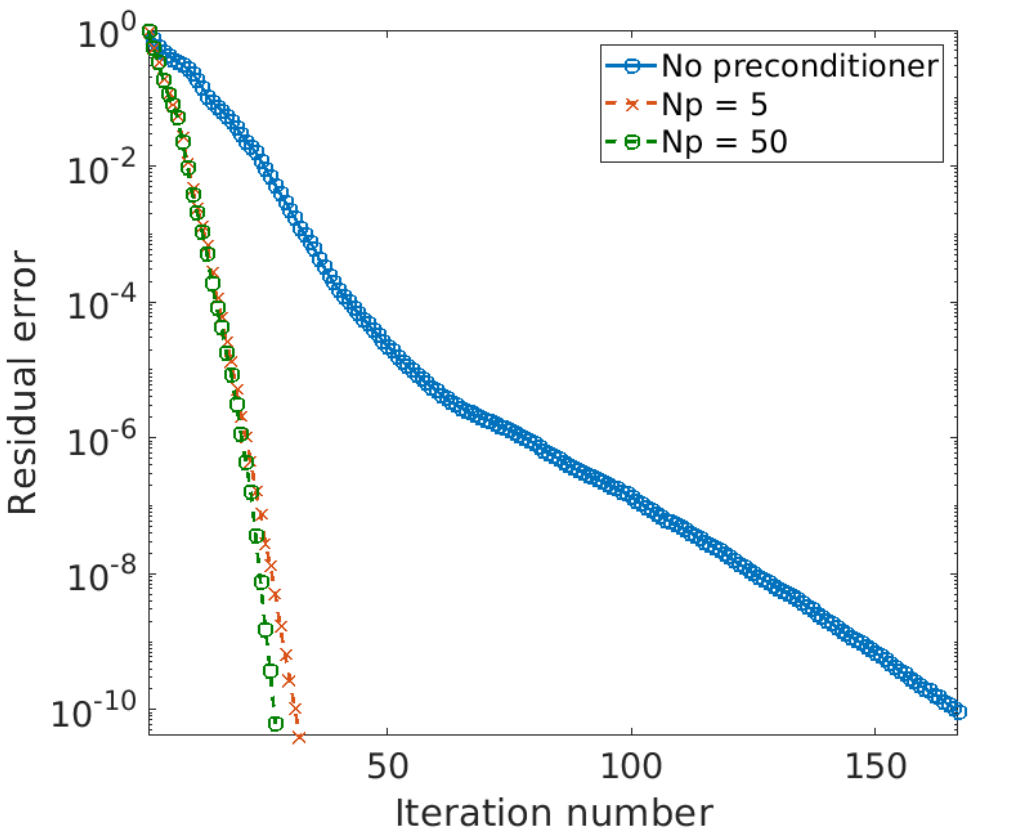}
	\end{subfigure} \hspace{0.1cm}
	\begin{subfigure}[t]{0.45\linewidth}
		\centering
		\includegraphics[height=3.9cm]{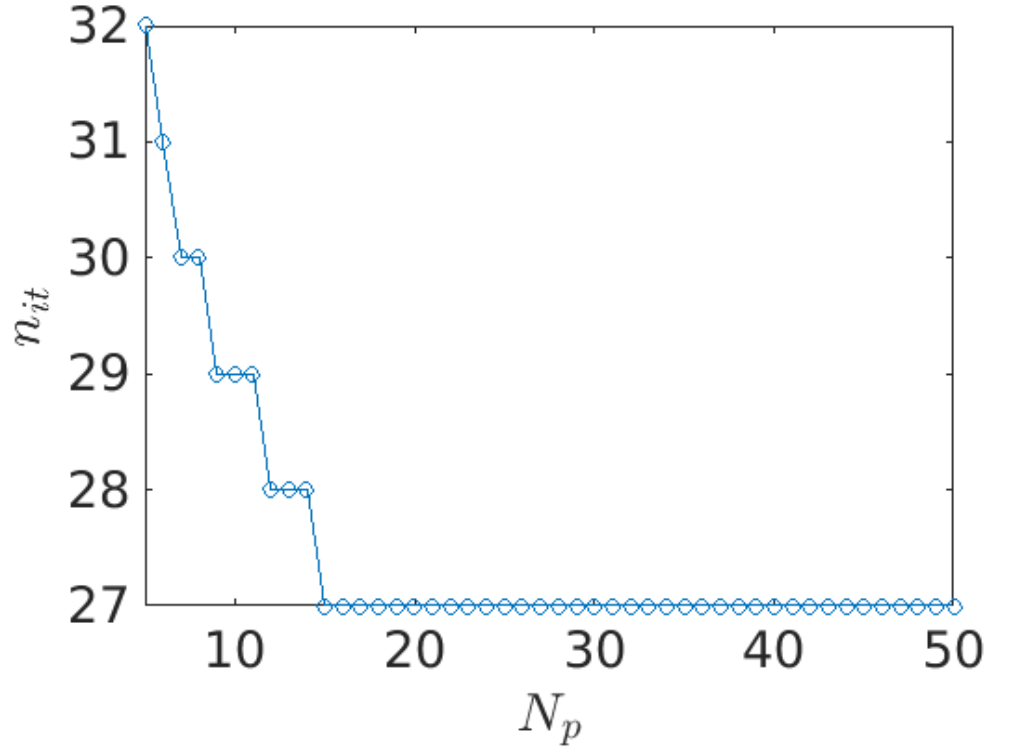}
	\end{subfigure}
	\caption{Influence of the number of Padé approximants $N_p$ on the number $n_{it}$ of GMRES iterations for the resolution of the Helmholtz weighted single-layer integral equation on the spiral-shaped screen for $k\abs{\Gamma}=200\pi$. The left figure compares the number of iterations when $N_p = 5$ and $N_p = 50$ (dashed line, respectively green circles and red crosses), to the case where no preconditioner is used (solid line, blue circles). On the right panel, we plot the number of iterations as a function of $N_p$ for this problem. One can see that the number of iterations stagnates once $N_p > 15$. }
	\label{sensibilitePade}
\end{figure}

\subsection{Importance of the correction}

It is crucial to include the correct dependence in $k$ in the preconditioners. We report here some numerical results in several situations where this dependence is not respected. 

\paragraph{Laplace preconditioning} First, we precondition the Helmholtz weighted single-layer integral equation on the flat segment with the operator
\[P'_0 = \sqrt{-(\omega \partial_x)^2 + I_d}\]
instead of 
\[P_k = \sqrt{-(\omega \partial_x)^2 -k^2 \omega^2}\,.\]
An identity operator is added under the square root to $P'_0$ to make it invertible. It is easy to check that $2P'_0$ is spectrally equivalent to the inverse of $S_{0,\omega}$ with 
\[\norm{2P'_0S_{0,\omega}} \leq  \sqrt{2} \leq 1.5 \,, \quad \norm{(2P'_0S_{0,\omega})^{-1}} \leq \frac{1}{\ln 2} \leq 1.5\,.\]
Since the Laplace preconditioner is also a compact perturbation of the inverse of $S_{k,\omega}$, the theory predicts (see e.g. \cite{hiptmair2006operator}) that the number of iterations remains bounded when the frequency is fixed and the mesh is refined. This result is confirmed numerically in \autoref{constantItKfixed}. We see indeed in practice that no matter how the mesh is refined, the number of iterations remains constant. However, we see in the previous example that this number of iterations grows with $k$. Including the dependence in $k$ in the preconditioner reduces a lot this behavior, as illustrated in \autoref{figureQuiTue}.

\begin{figure}[H]
	\centering
	\begin{subfigure}[t]{0.45\linewidth}
		\centering
		\includegraphics[height=3.3cm]{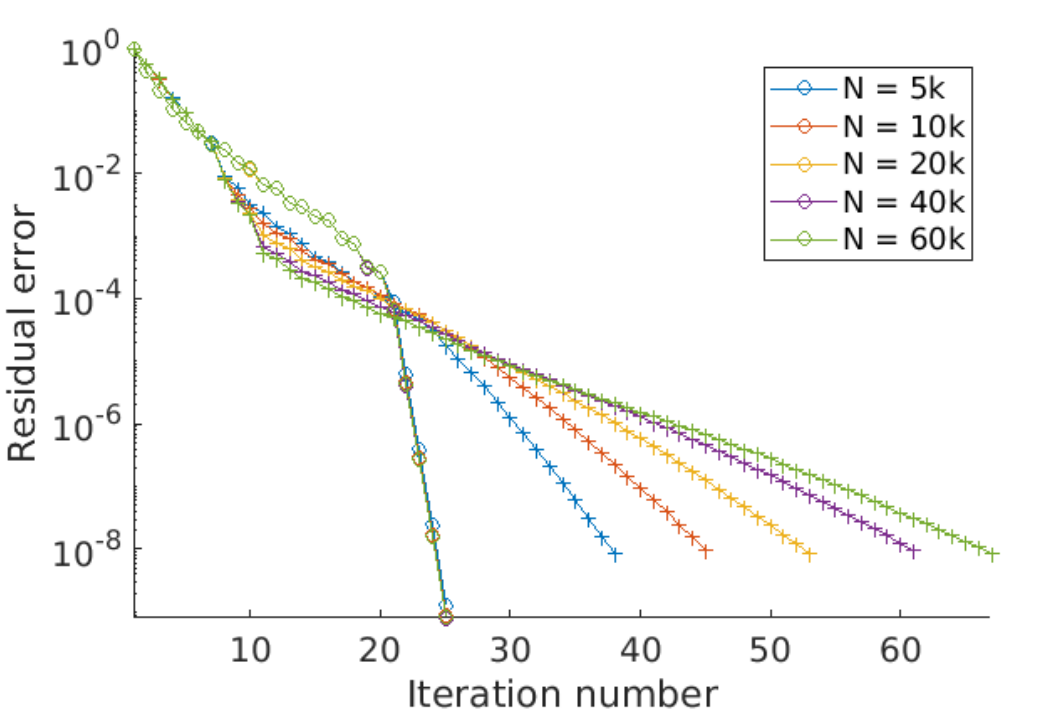}
	\end{subfigure}
	\begin{subfigure}[t]{0.45\linewidth}
		\centering
		\includegraphics[height=3.2cm]{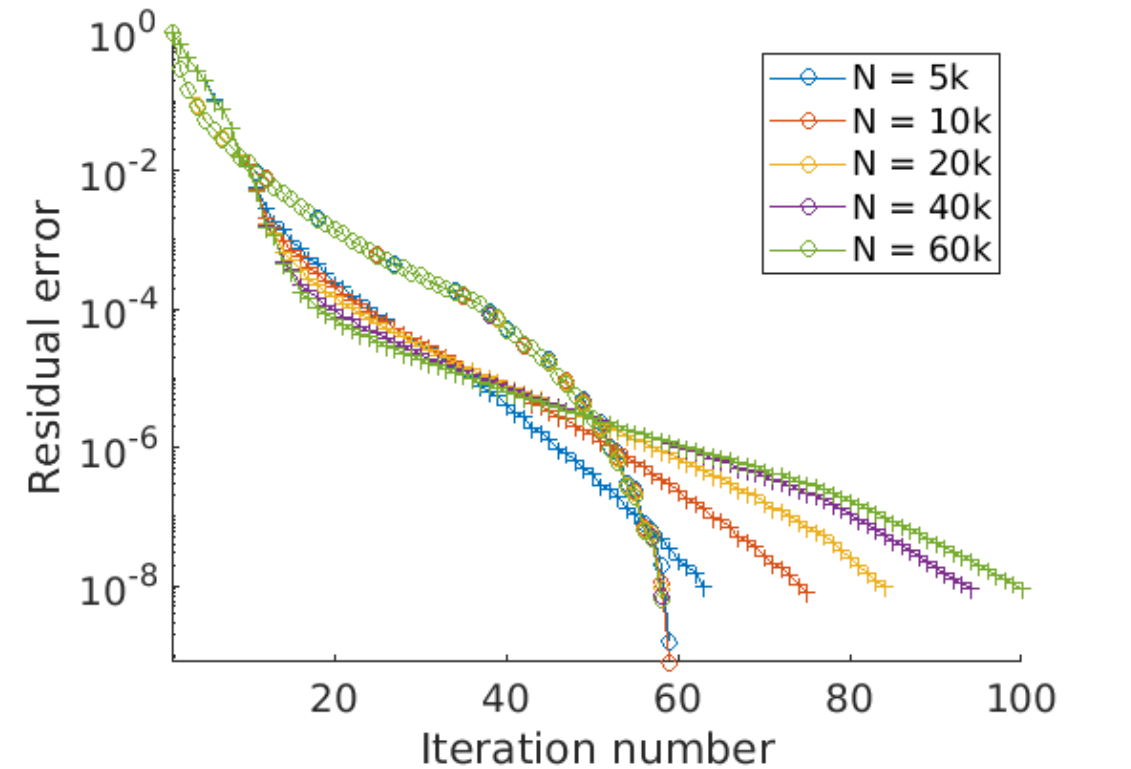}
	\end{subfigure}
	\caption{History of the GMRES relative residual for the resolution of the Helmholtz weighted single-layer integral equation on the flat segment for $k = 15$ (left) and $k = 50$ (right). The mesh is progressibely refined, the level of refinement begin indicated by the color of the curves. We compare the situation where the linear system is preconditioned by $P'_0$ (circles) as opposed to the case where no preconditioner is used (crosses). Notice that the curves with circles are almost superimposed. We thus verify in practice that the number of iterations for the preconditioned system is independent of the discretization (though not of $k$)}
	\label{constantItKfixed}
\end{figure}
\vspace{-1.3cm}
\begin{figure}[H]
	\centering
	\includegraphics[width= 0.5\textwidth]{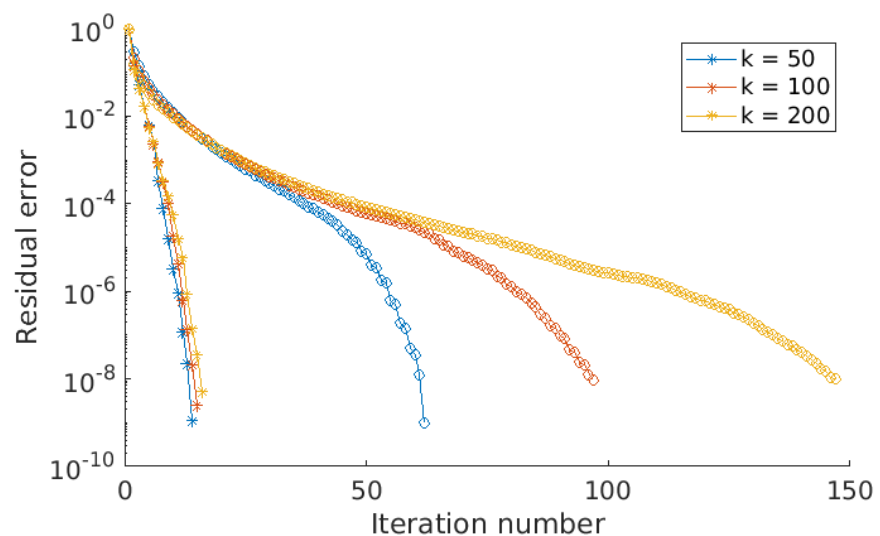}
	\caption{History of the GMRES relative residual for solving the Helmholtz weighted single-layer integral equation preconditioned either with $P'_0$ (circles) or $P_k$ (stars). Different colors represent different wave numbers. In each case, we keep the proportionality $N \approx 5k\abs{\Gamma}$. We notice that $P_k$ leads to similar numbers of iterations for the three wavenumbers. Such robustness is not observed for $P'_0$. }
	\label{figureQuiTue}
\end{figure}

\paragraph{No singularity correction} Second, we test the preconditioner without singularity correction, which is the method obtained when we take $\omega \equiv 1$. That is, we solve the non-weighted integral equation 
\begin{equation}
\label{nonWeigthed}
S_k \lambda = u_D
\end{equation}
on the flat segment with a standard $\mathbb{P}^1$ Galerkin method, and build a preconditioner based on the operator 
\begin{equation}
\label{defPprimeK}
P'_k = \sqrt{-\partial_x^2 - k^2 I_d}\,.
\end{equation}
This is the direct application of the method of Antoine and Darbas \cite{antoine2007generalized} to the context of an open curve. As stated at the beginning of section 4, if a uniform mesh is used, the Galerkin approximation converges at the rate $O(\sqrt{h})$ only. One remedy is to use a mesh graded towards the edges. A graded mesh of parameter $\beta$ is a mesh such that near the edge, the width of the $i$-th interval is approximately $(ih)^\beta$. The parameter $\beta = 2$ corresponds to the mesh defined in our Galerkin method defined, and $\beta = 5$ is the one that theoretically leads to the same rate of convergence as in our method \cite{postell1990h}. In \autoref{tableGraded}, we report the number of iterations in the GMRES method for the resolution of eq.~\eqref{nonWeigthed} preconditioned by $P'_k$ for $k = 10\pi$ on graded meshes for different parameters $\beta$. We compare the results with our preconditioned weighted Galerkin method. In each case, we report the $H^{-\frac{1}{2}}$ error. One can see that mesh-refinement allows to decrease the error at the price of losing the performance of the preconditioner. This justifies the need for the method introduced in this paper. 

\begin{table}[H]
	\begin{center}
		\begin{tabular}{m{10em}|m{3em}|m{12em}}
			\hline
			Numer. method & $n_{it}$ & relative $H^{-1/2}$ error \\
			\hline\hline
			Unif. mesh ($\beta = 1$) & 10 & 0.088 \\\hline
			Graded $\beta = 2$ & 12 &0.020 \\\hline
			Graded $\beta = 3$ & 13 &0.0066 \\\hline
			Graded $\beta = 4$ & 17 &0.0036 \\\hline
			Graded $\beta = 5$ & 21 &0.0030\\\hline
			Weighted Galerk. & 7 &2.2e-5\\\hline
		\end{tabular}
	\end{center}
	\caption{Number of GMRES iterations and $H^{-1/2}$ error for the resolution of the singe-
		layer integral equation $S_k \lambda = u_D$ on the flat segment, with $k = 10\pi$, $N = 80$. In the first five lines, the equation is solved using a standard Galerkin method on a graded mesh with
		parameter $\beta$ ranging from 1 (uniform mesh) to 5. In all those 5 cases, we precondition the linear system using a discrete version of $P'_k$ eq.~\eqref{defPprimeK}. The last line refers to the weighted Galerkin method described in this work, with the new square root preconditioner. Even for this small problem, the Galerkin matrices associated to graded meshes
		become very close to singular, and the preconditioners are no longer accurately evaluated using the Padé approximation as in section 5. The square root is instead computed directly using an eigenvalue decomposition of the matrices. The $H^{-1/2}$ error is obtained by computing a reference solution on a mesh which is refined until the error estimates stabilize.}		
	\label{tableGraded}
\end{table}

\subsection{Comparison with the generalized Calder\'{o}n preconditioners}

We finally adapt to our context the idea of Bruno and Lintner \cite{bruno2012second}, namely to use $S_{k,\omega}$ and $N_{k,\omega}$ as mutual preconditioners. Notice that the way we discretize the problem is different from \cite{bruno2012second} where a spectral method is used. In our setting, using the notation of section 5, we define the preconditioners
\[P_1 = [I_d]_\omega^{-1} [N_{k,\omega}]_\omega [I_d]_\frac{1}{\omega}^{-1}\,,\quad P_2 = [I_d]^{-1}_\frac{1}{\omega} [S_{k,\omega}]_\frac{1}{\omega}[I_d]_\omega^{-1}\]
respectively for the weighted single-layer and weighted hypersingular integral equations. We report the number of iterations and computing times respectively for the Dirichlet and Neumann problems on the flat segment respectively in \autoref{TableBrunoVsSqrtDir} and \autoref{TableBrunoVsSqrtNeu}. The performance is compared to that of our new preconditioners. The rhs are respectively $u_D = u_{inc}$ and $u_N = \frac{\partial u_{inc}}{\partial n}$ where $u_{inc}$ is a plan wave of angle of incidence $\frac{\pi}{4}$. Our results confirm the efficiency of the Generalized Calder\'{o}n preconditioners, which iteration counts remain very stable with respect to $k$. Despite a slightly larger increase of the iterations for our method, the resolution resolution remains faster for the tests presented here, particularly for the Dirichlet problem. This is due to the fact that our preconditioners are evaluated faster. 

\begin{table}[H]
	\begin{center}
		\begin{tabular}{m{4em} | m{4em} | m{4em} | m{4em} | m{4em}} 
			\hline
			\multicolumn{1}{c|}{ }&
			\multicolumn{2}{c|}{Calder{\'o}n Prec.}&\multicolumn{2}{c}{Square root Prec.}\\
			\hline
			$k \abs{\Gamma}$ & $n_{it}$& t(s) & $n_{it}$ & t(s)\\
			\hline\hline
			50$\pi$ & 15 & $<0.1$ & 8 & $< 0.1$\\
			\hline
			200$\pi$ & 15 & 0.45 & 10 &  0.35\\
			\hline
			400$\pi$ & 15 & 11 & 13 & 5\\
			\hline
			800$\pi$ & 15 & 42 & 16 & 18\\
			\hline
		\end{tabular}
	\end{center}
	\caption{Computing time and number of GMRES iterations for the Helmholtz weighted single-layer integral equation on the flat segment, respectively with the Calder\'{o}n preconditioner and with our new square-root preconditioners.}
	\label{TableBrunoVsSqrtDir}
\end{table}
\begin{table}[H]
	\begin{center}
		\begin{tabular}{m{4em} | m{4em} | m{4em} | m{4em} | m{4em}} 
			\hline
			\multicolumn{1}{c|}{ }&
			\multicolumn{2}{c|}{Calder{\'o}n Prec.}&\multicolumn{2}{c}{Square root Prec.}\\
			\hline
			$k \abs{\Gamma}$ & $n_{it}$& t(s) & $n_{it}$ & t(s)\\
			\hline\hline
			50$\pi$ & 15 & $<0.1$ & 10 & $< 0.1$\\
			\hline
			200$\pi$ & 16 & 0.3 & 13 &  0.3\\
			\hline
			400$\pi$ & 17 & 18 & 15 & 7\\
			\hline
			800$\pi$ & 17 & 68 & 18 & 34\\
			\hline
		\end{tabular}
	\end{center}
	\caption{Computing time and number of GMRES iterations for the Helmholtz weighted hypersingular integral equation on the flat segment, respectively with the Calder\'{o}n preconditioner and with our new square-root preconditioners.}
	\label{TableBrunoVsSqrtNeu}
\end{table}

\section{Conclusion} 
We have presented a new approach for the preconditioning of integral equations coming
from the discretization of wave scattering problems in 2D by open arcs. The methodology
is very effective and proven to be optimal for Laplace problems on straight segments. It generalizes the formulas mainly proposed in \cite{antoine2007generalized} for regular domains, by the simple addition of a suitable weight.
We deeply believe that the methodology opens new perspectives for such problems.
First, it is possible to generalize the approach in 3D for the diffraction by a disk (see \cite[Chap. 4]{these}). Second, the strategy that we used here seems very likely to be extended to
the half line and hopefully to 2D sectors, giving, on the one hand a new pseudo-differential
analysis more suitable than classical ones (see e.g. \cite{melrose,schulze1,schulze2}) for handling Helmholtz-like
problems on singular domains, and, on the other hand, a completely new preconditioning
technique adapted to the treatment of BEM operators on domains with corners or wedges
in 3D. Eventually, the weighted square root operators that appear in the present context might well be generalized to give suitable approximation of the exterior Dirichlet
to Neumann map for the Helmholtz equation. Having efficient approximations of this map is of particular importance in many contexts, such as e.g. domain decomposition methods.


\begin{thebibliography}{10}
	
	\bibitem{alouges2007stable}Alouges, F.,Borel, S., Levadoux, D.: 
	\newblock A Stable well conditioned integral equation for electromagnetism scattering, \newblock{\em J. Comput. Appl. Math.} 204(2), 440--451 (2007)
	
	\bibitem{alouges2005new}Alouges, F., Borel, S., Levadoux, D.:
	\newblock  A new well-conditioned integral formulation for Maxwell equations in three-dimensions. 
	\newblock{\em IEEE Trans. on Antennas and Propagation} 53(9), 2995-3004 (2005)
	
	\bibitem{antoine2007generalized}
	Antoine, X., Darbas, M.:
	\newblock Generalized combined field integral equations for the iterative
	solution of the three-dimensional helmholtz equation.
	\newblock {\em ESAIM: Mathematical Modelling and Numerical Analysis}
	41(1), 147--167 (2007)
	
	\bibitem{atkinson1991numerical}
	Atkinson, K.~E., Sloan, I.~H.:
	\newblock The numerical solution of first-kind logarithmic-kernel integral
	equations on smooth open arcs.
	\newblock {\em mathematics of computation} 56(193), 119--139 (1991)
	
	\bibitem{averseng}
	Averseng, M.:
	\newblock Pseudo-differential analysis of the Helmholtz layer potentials on open curves
	\newblock {\em arXiv preprint} arXiv:1905.13604, (2019)
	
	\bibitem{averseng2017}
	Averseng, M.:
	\newblock Fast discrete convolution in $\mathbb {R}^{2} $ with radial kernels using non-uniform fast Fourier transform with nonequispaced frequencies. 
	\newblock {\em Numerical Algorithms}, 1--24 (2019) 
	
	\bibitem{these}
	Averseng, M.:
	\newblock Efficient methods for scattering in 2D and 3D: preconditioning on singular domains and fast convolutions. 
	\newblock {\em Msc}, \'{E}cole {P}olytechnique (2019)
	
	\bibitem{axelsson2009equivalent}
	Axelsson, O., Kar{\'a}tson, J.:
	\newblock Equivalent operator preconditioning for elliptic problems.
	\newblock {\em Numerical Algorithms} 50(3), 297--380 (2009)
	
	\bibitem{betcke2014spectral}
	Betcke, T., Phillips, J., Spence, E.~A.:
	\newblock Spectral decompositions and nonnormality of boundary integral
	operators in acoustic scattering.
	\newblock {\em IMA Journal of Numerical Analysis} 34(2), 700--731, 2014.
	
	\bibitem{bruno2012second}
	Bruno, O.~P., Lintner, S.~K.:
	\newblock Second-kind integral solvers for TE and TM problems of diffraction by
	open arcs.
	\newblock {\em Radio Science} 47(6), 1--13 (2012)
	
	\bibitem{christiansen2002preconditioner}
	Christiansen, S.~H., N{\'e}d{\'e}lec., J.-C.:
	\newblock A preconditioner for the electric field integral equation based on
	calderon formulas.
	\newblock {\em SIAM Journal on Numerical Analysis} 40(3), 1100--1135 (2002).
	
	\bibitem{costabel1987improved}
	Costabel, M., Ernst, E.~P.:
	\newblock An improved boundary element Galerkin method for three-dimensional crack problems.
	\newblock {\em Integral Equations and Operator Theory} 10(4), 467--504 (1987).
	
	\bibitem{costabel1988convergence}
	Costabel, M., Ervin, V.~J.,Stephan, E.~P.:
	\newblock On the convergence of collocation methods for Symm's integral
	equation on open curves.
	\newblock {\em Mathematics of computation} 51(183):167--179 (1988)
	
	\bibitem{costabel2003asymptotics}
	Costabel, M., Dauge, M., Duduchava, R.:
	\newblock Asymptotics without logarithmic terms for crack problems.
	\newblock (2003)
	
	\bibitem{darbas2004preconditionneurs}
	Darbas, M.:
	\newblock Pr{\'e}conditionneurs Analytiques de type Calder{\`o}n pour les Formulations Int{\'e}grales des Probl{\`e}mes de Diffraction d'ondes
	\newblock {\em Msc}, INSA Toulouse (2004)
	
	\bibitem{djikstra}
	Djikstra, W., Hochstenbach, M.~E.: 
	\newblock Numerical approximation of the logarithmic capacity. 
	\newblock {\em CASA report} (2008)
	
	
	\bibitem{estrada1989integral}
	Estrada, R. and Kanwal, R.~P.:
	\newblock Integral equations with logarithmic kernels
	\newblock {\em IMA Journal of Applied Mathematics} 43(2), 133--155 (1989)
	
	
	\bibitem{gimperlein2019optimal}
	Gimperlein, H., Stocek, J., Urzua-Torres, C:
	\newblock Optimal operator preconditioning for pseudodifferential boundary problems
	\newblock {\em arXiv preprint} arXiv:1905.03846 (2019)
	
	\bibitem{greengard1987fast}
	Greengard, L., Rokhlin, V.:
	\newblock A fast algorithm for particle simulations
	\newblock {\em Journal of computational physics} 73(2), 325--348 (1987)
	
	
	\bibitem{hale2008computing}
	Hale, N., Higham, N.~A., Trefethen, L.~N.:
	\newblock Computing $A^{\alpha}$, $\log(A)$, and related matrix
	functions by contour integrals.
	\newblock {\em SIAM Journal on Numerical Analysis} 46(5), 2505--2523 (2008)
	
	\bibitem{hall2013quantum}
	Hall, B.~C.:
	\newblock {\em Quantum theory for mathematicians}. Graduate Texts in Mathematics 267 (2013)
	
	\bibitem{hiptmair2006operator}
	Hiptmair., R.:
	\newblock Operator preconditioning.
	\newblock {\em Computers and mathematics with Applications} 52(5), 699--706 (2006)
	
	\bibitem{hiptmair2014mesh}
	Hiptmair, R., Jerez-Hackes, C., Urz{\'u}a~Torres, C.~A.:
	\newblock Mesh-independent operator preconditioning for boundary elements on
	open curves.
	\newblock {\em SIAM Journal on Numerical Analysis} 52(5), 2295--2314 (2014)
	
	\bibitem{hiptmair2017closed}
	Hiptmair, R., Jerez-Hanckes, C., Urz{\'u}a~Torres, C.:
	\newblock Closed-form inverses of the weakly singular and hypersingular operators on disks. 
	\newblock {\em Integral Equations and Operator Theory} 90(1), 4 (2018)
	
	\bibitem{hormander2007analysis}
	H\"{o}rmander., L.:
	\newblock {\em The analysis of linear partial differential operators III: Pseudo-differential operators.}
	\newblock Springer Science \& Business Media (2007)
	
	\bibitem{jerez2012explicit}
	Jerez-Hanckes, C., N{\'e}d{\'e}lec, J.~C.:
	\newblock Explicit variational forms for the inverses of integral logarithmic
	operators over an interval.
	\newblock {\em SIAM Journal on Mathematical Analysis} 44(4), 2666--2694 (2012)
	
	\bibitem{jiang2004second}
	Jiang, S., Rokhlin, V.:
	\newblock Second kind integral equations for the classical potential theory on
	open surfaces II.
	\newblock {\em Journal of Computational Physics} 195(1),1--16 (2004)
	
	\bibitem{levadoux2001etude}
	Levadoux, D.:
	\newblock Etude d'une {\'e}quation int{\'e}grale adapt{\'e}e {\`a} la r{\'e}solution hautes fr{\'e}quences de l'{\'e}quation d'Helmholtz.
	\newblock {\em Msc}, Université Paris 6 (2001)
	
	\bibitem{mason2002chebyshev}
	Mason, J.~C., Handscomb, D.C.:
	\newblock {\em Chebyshev polynomials}.
	\newblock CRC Press (2002)
	
	\bibitem{mclean2000strongly}
	McLean, W.~C.~H.:
	\newblock {\em Strongly elliptic systems and boundary integral equations}.
	\newblock Cambridge university press (2000)
	
	\bibitem{melrose} Melrose, R.:
	\newblock Transformation of boundary problems. 
	\newblock{\em Acta Mathematica} 147, 149--236 (1981)
	
	
	\bibitem{monch1996numerical}
	M{\"o}nch, L.:
	\newblock On the numerical solution of the direct scattering problem for an
	open sound-hard arc.
	\newblock {\em Journal of computational and applied mathematics}
	71(2), 343--356 (1996)	
	
	\bibitem{NIST:DLMF}
	Olver, F.~W.~J., Olde~Daalhuis, A.~B., Lozier, D.~W., Schneider, B.~I.,	Boisvert, R.~F., Clark, C.~W., Miller, B.~R., Saunders, B.~V.:
	\newblock {\it NIST Digital Library of Mathematical Functions}.
	\newblock http://dlmf.nist.gov/, Release 1.0.16 of 2017-09-18.
	
	\bibitem{postell1990h}
	Postell, F.~V., Stephan, E.~P.:
	\newblock On the h-, p- and hp versions of the boundary element
	method : numerical results.
	\newblock {\em Computer Methods in Applied Mechanics and Engineering} 
	83(1), 69--89 (1990)
	
	\bibitem{ramaciotti2017some}
	Ramaciotti, P., N{\'e}d{\'e}lec, J.-C.:
	\newblock About some boundary integral operators on the unit disk related to
	the laplace equation.
	\newblock {\em SIAM Journal on Numerical Analysis} 55(4),1892--1914 (2017)
	
	\bibitem{schulze1} Rempel, S., Schulze, B.:
	\newblock Parametrices and boundary symbolic calculus for elliptic boundary problems without the transmission property.
	\newblock{\em Math. Nachr.} 105, 45--149 (1982)	
	
	\bibitem{schulze2} Rempel, S., Schulze, B.: 
	\newblock Asymptotics for elliptic mixed boundary problems. Pseudo-differential and Mellin operators in spaces with conormal singularity. 
	\newblock{\em  Mathematical Research} 50 (1989)
	
	\bibitem{saad1986gmres} Saad, Y., Schultz, M.~H.: 
	\newblock GMRES: A generalized minimal residual algorithm for solving nonsymmetric linear systems.
	\newblock{\em SIAM J. Sci. Stat. Comput.} 7, 856--869 (1986)	
	
	\bibitem{sauter2011boundary}
	Sauter, S.~A., Schwab, C.:
	\newblock {\em Boundary Element Methods}. Springer, Berlin, Heidelberg (2010)	
	
	\bibitem{sloan1992collocation}
	Sloan, I.~H., Stephan, E.~P.:
	\newblock Collocation with chebyshev polynomials for Symm's integral equation
	on an interval.
	\newblock {\em The ANZIAM Journal} 34(2), 199--211 (1992)
	
	\bibitem{steinbach1998construction}
	Steinbach, O., Wendland, W.~L.:
	\newblock The construction of some efficient preconditioners in the boundary
	element method.
	\newblock {\em Advances in Computational Mathematics} 9(1-2), 191--216 (1998)
	
	\bibitem{stephan1984augmented}
	Stephan, E.~P., Wendland, W.~L.:
	\newblock An augmented galerkin procedure for the boundary integral method
	applied to two-dimensional screen and crack problems.
	\newblock {\em Applicable Analysis} 18(3), 183--219 (1984) 
	
	\bibitem{wendland1990hypersingular}
	Stephan, E.~P., Wendland, W.~L.:
	\newblock A hypersingular boundary integral method for two-dimensional screen
	and crack problems.
	\newblock {\em Archive for Rational Mechanics and Analysis} 112(4), 363--390 
	(1990)
	
	\bibitem{urzua2014optimal}
	Urz{\'u}a~Torres, C.~A.:
	\newblock Optimal preconditioners for solving two-dimensional fractures and
	screens using boundary elements.
	\newblock {\em Msc},
	\newblock Pontifica universidad catolica de Chile (2014)
	
	\bibitem{yan1990cosine}
	Yan, Y.:
	\newblock Cosine change of variable for {S}ymm's integral equation on open arcs.
	\newblock {\em IMA Journal of Numerical Analysis} 10(4), 521--535 (1990)
	
	
\end{thebibliography}
\end{document}